\documentclass[12pt]{amsart}
\usepackage{array,amscd,amsmath,amsthm}
\usepackage{amssymb}

\usepackage{supertabular,multicol}
\usepackage{graphicx}
\usepackage{pstricks,psfrag}

\usepackage[all]{xy}
\allowdisplaybreaks

\makeatletter

\newenvironment{figurehere}
  {\def\@captype{figure}}
  {}
\makeatother

\makeatletter
\renewenvironment{proof}[1][\proofname]{\par
  \pushQED{\qed}%
  \normalfont \topsep6\p@\@plus6\p@\relax
  \trivlist
  \item[\hskip\labelsep
        \itshape
    #1\@addpunct{.}]
}{%
  \popQED\endtrivlist\@endpefalse
\mbox{}\medskip
}
\makeatother

\makeatletter
\let\mcnewpage=\newpage
\newcommand{\Trick}{
\renewcommand\newpage{%
        \if@firstcolumn
            \hrule width\linewidth height0pt
                \columnbreak
        \else
                \mcnewpage
        \fi
}
}
\makeatother

\newtheorem{theorem}{Theorem}[section]
\newtheorem{proposition}[theorem]{Proposition}
\newtheorem{lemma}[theorem]{Lemma}

\newtheorem{corollary}[theorem]{Corollary}
\newtheorem{question}[theorem]{Question}

\newtheorem*{proposition*}{Proposition}
\newtheorem*{theorem*}{Theorem}

\newtheorem{maintheorem}{Theorem}

\theoremstyle{definition}

\newtheorem{definition}[theorem]{Definition}
\newtheorem{remark}[theorem]{Remark}
\newtheorem*{definition*}{Definition}
\newtheorem*{notation}{Notation}

\newtheorem*{remark*}{Remark}

\theoremstyle{remark}

\newcommand{\CC}{\mathbb{C}}
\newcommand{\ZZ}{\mathbb{Z}}
\newcommand{\QQ}{\mathbb{Q}}
\newcommand{\RR}{\mathbb{R}}
\newcommand{\NN}{\mathbb{N}}

\def\PP{\mathbb{P}}

\def\dis{\displaystyle}

\def\exh{\xi}

\def\pa{\partial}

\def\cohdim{\mathrm{coh\text{-}dim}}

\def\dim{\mathrm{dim}}
\def\ord{\mathrm{ord}}

\def\aut{\mathrm{Aut}}

\def\ev{\mathrm{ev}}
\def\Ev{\mathrm{Ev}}

\def\a{\alpha}

\def\e{\varepsilon}

\def\ph{\varphi}

\def\rar{\rightarrow}
\def\surj{\twoheadrightarrow}
\def\arr#1#2{\stackrel{#1}{#2}}
\def\hra{\hookrightarrow}
\def\lra{\longrightarrow}

\def\ul#1{\underline{#1}}
\def\ol#1{\overline{#1}}

\def\wti#1{\widetilde{#1}}

\def\Ocal{\mathcal{O}}
\def\M{\mathcal{M}}

\def\C{\mathcal{C}}

\def\Sfrak{\mathfrak{S}}
\def\Wfrak{\mathfrak{W}}
\def\Pcal{\mathcal{P}}

\def\bm#1{\text{\boldmath$#1$}}

\def\Bcal{\mathcal{B}}

\def\Dcal{\mathcal{D}^\bullet}

\def\Tcal{\mathcal{T}^\bullet}
\def\Kcal{\mathcal{K}^\bullet}
\def\Hbb{\mathbb{H}}

\def\Vfrak{\mathfrak{V}}

\def\diskrm{N}
\def\disk{\mathcal{\diskrm}}

\def\Hloc{\bm{H}}

\def\h{\mathfrak{h}}
\def\eps{\varsigma}

\def\cdom{\Xi}

\def\mo{o}

\def\vc#1{{\bm{#1}}}
\def\num#1{{\bm{#1}}}

\begin{document}


\title[On the cohomological dimension of $\M_g$]{On the cohomological dimension \\ of the moduli space of Riemann surfaces}

\author{Gabriele Mondello}
\email{mondello@mat.uniroma1.it}
\address{Dipartimento di Matematica, ``Sapienza'' Universit\`a di Roma\\
Piazzale Aldo Moro 5 - 00185 Roma, Italy}
\thanks{The author's research was partially supported
by grant FIRB 2010 (RBFR10GHHH\_003) ``Low-dimensional geometry and topology''.}


\keywords{Riemann surfaces, moduli space, translation surfaces, cohomological dimension}
\subjclass{32G15, 32F10, 30F30}

\begin{abstract}
The moduli space of Riemann surfaces of genus $g\geq 2$ is (up to a finite \'etale cover) a complex manifold and so it makes
sense to speak of its Dolbeault cohomological dimension.
The conjecturally optimal bound is $g-2$. This expectation is verified in low genus and supported by Harer's computation
of its de Rham cohomological dimension and by vanishing results in the tautological intersection ring.
In this paper we prove that such dimension is at most $2g-2$.
We also prove an analogous bound for the moduli space of Riemann surfaces with marked points.
The key step is to show that the Dolbeault cohomological dimension of each stratum of translation surfaces is at most $g$.
In order to do that, we produce an exhaustion function whose complex Hessian has controlled index: the construction of
such a function relies on
some basic geometric properties of translation surfaces.
\end{abstract}

\maketitle

\tableofcontents

\section{Introduction}

\subsection{Cohomological dimensions of $\M_{g,n}$}

The moduli space $\M_{g,n}$ of compact connected Riemann surfaces of genus
$g\geq 2$ with $n\geq 0$ distinct marked points is an orbifold and so it makes sense to speak of its
de Rham cohomological dimension $\cohdim_{dR}(\M_{g,n})$, that is the greatest degree for which its
de Rham cohomology with coefficients in some flat vector bundle does not vanish.

In a similar fashion,
as $\M_{g,n}$ can also be given the structure of a complex-analytic orbifold,
it makes sense to speak of Dolbeault cohomology
with coefficients in a holomorphic vector bundle and of
Dolbeault cohomological dimension $\cohdim_{Dol}$
(see Section \ref{subsection:Dolbeault} for a precise definition).

Since $\M_{g,n}$ is irreducible and not compact, clearly 
$\cohdim_{dR}(\M_{g,n})<\mathrm{dim}_{\RR}(\M_{g,n})=6g-6+2n$ and
$\cohdim_{Dol}(\M_{g,n})<\mathrm{dim}_{\CC}(\M_{g,n})=3g-3+n$
but the actual cohomological dimensions could be much smaller.

Indeed, using topological methods, 
Harer \cite{harer:virtual} proved that $\cohdim_{dR}(\M_{g,n})=4g-5+n+\eps_n$, where $\eps_0=0$ and
$\eps_n=1$ for $n>0$.

On the other hand, for Dolbeault cohomology the problem has not been settled yet.
Though Harer's theorem above implies $\cohdim_{Dol}(\M_{g,n})\geq g-2+\eps_n$ (see Section \ref{sec:dR-cohdim} below),
it seems that at present the estimate 
$g-2+\eps_n\leq \cohdim_{Dol}(\M_{g,n})
\leq 3g-4+\eps_n$ is the only available one for $g>5$.

\begin{question}[Looijenga]
Is $\cohdim_{Dol}(\M_{g,n})=g-2+\eps_n$ ?
\end{question}

The aim of this paper is to make some progress in the direction
of Looijenga's question: the following is our main result.

\begin{maintheorem}\label{mthm:cohdim}
For every $g\geq 2$ and $n\geq 0$, the Dolbeault cohomological dimension of the moduli space of Riemann surfaces satisfies
\[
\cohdim_{Dol}(\M_{g,n})\leq (g-2+\eps_n)+g\, .
\]
\end{maintheorem}

\begin{remark}
The moduli space $\M_{g,n}$ is also a Deligne-Mumford stack
and so also cohomology with coefficients in an algebraic coherent sheaf can be considered and an ``algebraic cohomological dimension'' $\cohdim_{alg}$ of $\M_{g,n}$ can be similarly defined.
Looijenga phrased an analogous conjecture in this setting,
namely whether $\cohdim_{alg}(\M_{g,n})= g-2+\eps_n$.
However, in this paper we will only deal with Dolbeault cohomological dimensions.
\end{remark}

\subsection{Elementary properties of Dolbeault cohomological dimension}

Even though Serre's GAGA theorems \cite{serre:gaga} establish an isomorphism between the cohomology of an 
algebraic coherent sheaf 
on a complex projective manifold and of its associated analytic coherent sheaf,
the quantities
$\cohdim_{alg}$ and $\cohdim_{Dol}$ need not agree when working
with a non-compact algebraic variety or Deligne-Mumford stack.
Serre himself produced an example of a smooth quasi-projective
complex surface $S$ with $\cohdim_{Dol}(S)=0$ but $\cohdim_{alg}(S)=1$, being $S$ Stein but not affine (see \cite{hartshorne:ample}, Chapter VI, Example 3.2).

Most naive considerations about algebraic and Dolbeault cohomological
dimensions are a consequence of a few basic properties of these invariants,
listed here below.


\begin{itemize}
\item[(P1)]
A connected complex manifold $X$ is compact if and only if $\cohdim_{Dol}(X)=\dim_{\CC}(X)$; analogously, a smooth connected algebraic variety is complete if and only if $\cohdim_{alg}(X)=\dim_{\CC}(X)$ (both statements are a consequence of Serre
duality).
\item[(P2)]
An algebraic manifold $X$ is affine if and only if $\cohdim_{alg}(X)=0$;
analogously, a complex manifold $X$ is Stein if and only if $\cohdim_{Dol}(X)=0$ (Serre \cite{serre:fac}). Thus, $\cohdim_{alg}(X)=0\implies \cohdim_{Dol}(X)=0$, because
algebraic affine manifolds are Stein.
\item[(P3)]
If $Y\rar X$ is an affine algebraic morphism, then
$\cohdim_{alg}(Y)\leq\cohdim_{alg}(X)$
and $\cohdim_{Dol}(Y)\leq\cohdim_{Dol}(X)$
(see Lemma \ref{lemma:fibr}(c)).
\item[(P4)]
If $f:Y\rar X$ is a finite holomorphic map of complex manifolds, then
$\cohdim_{Dol}(Y)\leq\cohdim_{Dol}(X)$; moreover, equality holds if
$f$ is finite surjective (see Lemma \ref{lemma:fibr}(a)).
An analogous statement holds for $\cohdim_{alg}$, if $f$ is a finite morphism of complex
algebraic varieties. 
\item[(P5)]
If $f:Y\rar X$ is a proper surjective holomorphic submersion of relative dimension $r$,
then $\cohdim_{Dol}(Y)=\cohdim_{Dol}(X)+r$ (see Lemma \ref{lemma:fibr}(b)).
Analogous statement for $\cohdim_{alg}$, if $f$ a smooth proper surjective morphism between complex algebraic varieties.
\item[(P6)]
If $\Vfrak=\{V_\sigma\}$ is an open cover of $X$, then
$\cohdim_{Dol}(X)$ is bounded above by the maximum of
$\cohdim_{Dol}(V_{\sigma_0}\cap\dots\cap V_{\sigma_d})+d$, ranging over
non-empty intersections of distinct open subsets of $\Vfrak$ (see Lemma \ref{lemma:local-global}). The same statement for $\cohdim_{alg}$ holds, if $\Vfrak$ is an open cover for
the Zariski topology of the complex algebraic variety $X$. 
\item[(P7)]
If $X=X^0\supset X^1\supset\dots\supset X^k\supset X^{k+1}=\emptyset$ is a stratification
by closed algebraic subvarieties such that $X^{d+1}$ is a Cartier divisor
inside $X^d$, then $\cohdim_{alg}(X)$ is bounded above by the maximum of $\cohdim_{alg}(X^d\setminus X^{d+1})+d$ (see below).
\item[(P8)]
If $E\rar X$ is a holomorphic vector bundle on the complex manifold $X$ and $c_q(E)\neq 0$
in $H^{2q}(X;\CC)$, then $\cohdim_{Dol}(X)\geq q$.
If moreover $E$ and $X$ are algebraic, then $\cohdim_{alg}(X)\geq q$. 
(Both results follow by noticing that $c_q(E)$ can be lifted along
$H^{0,q}_{\ol{\pa}}(X;\Omega^{q,0}_X)\rar H^{2q}(X;\CC)$ in the former case
and along $H^q(X;\Omega^q_X)\rar H^{2q}(X;\CC)$ in the latter).
\end{itemize}

A first glance at (P3) and (P5) above immediately leads us to deduce that
$\cohdim_{Dol}(\M_{g,1})=\cohdim_{Dol}(\M_g)+1$
and $\cohdim_{Dol}(\M_{g,n})\leq \cohdim_{Dol}(\M_g)+1$ for $n>1$
(and the same holds for $\cohdim_{alg}$).
In fact, the map $\M_{g,1}\rar\M_g$ that forgets the marked point is algebraic, smooth and proper of relative dimension $1$; whereas the map
$\M_{g,n+1}\rar\M_{g,n}$ that forgets the $(n+1)$-th marked point is algebraic affine, provided $n>0$.

These properties are discussed in more detail in Appendix \ref{sec:appendix}.

\begin{remark}\label{rmk:property7}
Property (P7) highlights a difference between $\cohdim_{alg}$ and $\cohdim_{Dol}$, the reason being that meromorphic functions 
admit a Laurent expansion and so they always leave a trace on their polar locus, namely a section of some algebraic coherent sheaf; by contrast, analytic functions can have essential singularities. The existence of such a ``trace'' in the algebraic setting is encoded in the long exact sequence for local cohomology (see \cite{sga2}).
In order to stress such difference and justify why we will estimate $\cohdim_{Dol}$ using (P6), we give below a quick proof of (P7).
\end{remark}

\begin{proof}[Proof of property (P7)]
We proceed by induction on $k\geq 0$.
The case $k=0$ being trivial, we assume the result true for stratifications with $k+1$ layers.

Let $X=X^0\supset\cdots\supset X^{k+1}\supset X^{k+2}=\emptyset$ a stratification by $k+2$ layers as in the hypotheses.
By induction applied to $X^1$, we have $\cohdim_{alg}(X^1)\leq
\max_{d\geq 1} \cohdim_{alg}(X^d\setminus X^{d+1})+d-1$.

Let $\mathcal{F}$ be an algebraic coherent sheaf on $X$. The local cohomology sheaf $\mathcal{H}^s_{X^1}(\mathcal{F})$ on $X$ is
quasi-coherent and supported on $X^1$ for $s=0,1$; moreover, it vanishes for $s>1$, because $X_1$ is a Cartier divisor in $X_0$
(for instance, see \cite{sga2}, Expos\'e II, Corollaire 4).
Moreover, the local-to-global spectral sequence $H^r(X;\mathcal{H}^s_{X^1}(\mathcal{F}))\implies H^{r+s}_{X^1}(X;\mathcal{F})$ implies that
$H^q_{X^1}(X;\mathcal{F})=0$ for $q>\cohdim_{alg}(X^1)+1$.
Finally, the long exact sequence
\[
\dots \rar H^q_{X^1}(X;\mathcal{F})\rar H^q(X;\mathcal{F})
\rar H^q(X\setminus X^1;\mathcal{F})\rar \dots
\]
shows that $H^q(X;\mathcal{F})=0$ for $q>\max\{\cohdim_{alg}(X^1)+1,\cohdim_{alg}(X\setminus X^1)\}$, that is for
$q>\max_{d\geq 0}\{\cohdim_{alg}(X^d\setminus X^{d+1})+d\}$. This concludes the proof.
\end{proof}

\subsection{Known results and evidence for the conjecture}

\subsubsection{Low genera cases}
The situation for the moduli spaces $\M_{0,n}$
of Riemann surfaces of genus $0$ with $n\geq 3$ distinct marked points and for $\M_{1,n}$ with $n\geq 1$ is completely understood.
Indeed, $\M_{0,n}$ is isomorphic to
$\{z\in (\CC\setminus\{0,1\})^{n-3}\,|\, z_i\neq z_j\ \text{if $i\neq j$}\}$ and so it is an affine algebraic variety; on the other hand,
$\M_{1,1}$ is dominated by $\M_{0,4}$ through an \'etale cover of degree $6$.
Hence, they have algebraic and so Dolbeault cohomological dimension $0$, and so does $\M_{1,n}$.

Since all Riemann surfaces of genus $2$ are hyperelliptic,
a finite \'etale cover of $\M_2$ is isomorphic to $\M_{0,6}$ and so again is algebraic affine. Thus, algebraic and Dolbeault cohomological dimensions are $0$ for $\M_2$, they are both $1$
for $\M_{2,1}$ and $\leq 1$ for $\M_{2,n}$ with $n>1$.

Similar considerations show that the hyperelliptic locus in $\M_g$ has always $\cohdim_{alg}=0$. Since non-hyperelliptic Riemann surfaces of genus $3$ can be canonically realized as smooth plane quartics (up to action of $\mathrm{PGL}_3$), the complement of the Cartier divisor consisting of hyperelliptic curves in $\M_3$ is the quotient of an affine variety by a reductive group and so it is affine.
Hence, $\cohdim_{alg}(\M_3)\leq 1$ by (P7).
Since it can be shown that $\M_3$ contains a compact Riemann surface (because its Satake compactification is projective and with boundary of codimension $2$), necessarily $\cohdim_{alg}(\M_3)=1$.

In order to deal with the cases of genus $4$ and $5$, a finer analysis is needed. This is done for $n=0$ in Fontanari-Looijenga \cite{fontanari-looijenga}
and in Fontanari-Pascolutti \cite{fontanari-pascolutti}, where an affine stratification
(using canonical models)
and an affine open cover (using $\theta$-functions) respectively are exhibited.

Thus, Looijenga's conjecture is settled for $g\leq 5$.

\subsubsection{De Rham cohomological dimension}\label{sec:dR-cohdim}

Harer's proof \cite{harer:virtual} that 
$\cohdim_{dR}(\M_{g,n})=4g-5+\eps_n+n$ for $g\geq 1$ is truly topological and it does not seem possible to mimick it in the complex-analytic setting: it exploits two different but closely related versions of a cellularization of the moduli space
via ribbon graphs due to Harer-Mumford-Penner-Thurston (see \cite{bowditch-epstein:natural},
\cite{penner:decorated}, \cite{kontsevich:intersection}
and also \cite{mondello:triangulation}).

A link between de Rham and Dolbeault cohomological dimension
of $X$ is provided by a twisted version of the 
Hodge-Fr\"olicher spectral sequence (see Lemma \ref{lemma:Dol-dR}),
which ensures that
\[
\cohdim_{dR}(X)\leq\cohdim_{Dol}(X)+\mathrm{dim}_{\CC}(X).
\]
%
%
Thus, Harer's result implies that 
$\cohdim_{Dol}(\M_{g,n})\geq g-2+\eps_n$ for $g\geq 1$,
and so the bound conjectured by Looijenga is the smallest possible.
Conversely, proving Looijenga's conjecture would imply
Harer's upper bound for $\cohdim_{dR}(\M_{g,n})$.

\subsubsection{Complete subvarieties and stratifications}

The idea that the geometry of $\M_{g}$ could be better understood with the aid of geometrically meaningful stratifications is already in Arbarello \cite{arbarello}, who studied the so-called Weierstrass stratification, consisting of $g-1$ layers. By (P7), affineness of such locally closed strata would have implied Looijenga's conjecture. Even though the smallest stratum is affine
because it coincides with the hyperelliptic locus
and the top-dimensional stratum is also affine (as remarked in \cite{fontanari}), this turns out not to be the case in general
\cite{arbarello-mondello}.

A suitable modification of Arbarello's stratification was employed by Diaz \cite{diaz:complete} to show that a compact holomorphic subvariety of $\M_g$ has dimension at most $g-2$. Notice that this result too would be a consequence of Looijenga's conjecture by (P4).
We also remark that Diaz's bound is known to be attained for $g=2$ for trivial reasons and for $g=3$ because $\M_3$ contains a complete Riemann surface, as remarked above; on the other hand, what happens for $g\geq 4$ is still unknown.

The idea (already present in Arbarello's paper) behind Diaz's proof is to show that each locally closed stratum is quasi-affine and so it cannot contain a compact holomorphic subvariety of positive dimension.
However, since there is no control on the cohomological dimension of such quasi-affine locally closed layers, such a stratification does not seem to say much about Looijenga's question.

\subsubsection{Stratifications and tautological classes}

Looijenga \cite{looijenga:tautological} exploited a variant of Diaz's stratification to show that certain characteristic classes on $\M_g$
(resp. on $\M_{g,1}$) of type $(q,q)$, also called ``tautological classes'', 
vanish in 
degrees $q>g-2$ (resp. $q>g-1$).
A key point of the proof is to show that, on the quasi-affine locally closed strata,
(suitable powers of) all the involved tautological line bundles trivialize and so their Chern classes vanish.

Looijenga's question formulated at the beginning is clearly an
amplification of this
vanishing theorem (proven to be sharp by Faber in \cite{faber}).
Further extensions of the same question
have also been formulated for suitable
partial compactifications of $\M_{g,n}$, such as the moduli space of
Riemann surfaces of compact type, 
or of irreducible Riemann surfaces,
or with rational tails, or with
at most $k$ rational components. 
For the last family of partially compactified moduli spaces, 
Ionel \cite{ionel} and Graber-Vakil \cite{graber-vakil} proved the analogous of Looijenga's vanishing; the parallel topological statements for de Rham cohomology follow from Harer's work and were analyzed in \cite{mondello:homotopical}.

Despite the relevant amount of geometric information that the intersection theory of such tautological classes carry, it is not clear whether such vanishing results allow to draw any conclusion on the cohomological dimension of $\M_{g,n}$.\\

The rather classical idea underlying all the above stratifications was to look at special classes of ramified covers of $\CC\PP^1$, which can be seen as meromorphic differentials with zero periods.
Using meromorphic differentials with real periods, Grushevsky-Krichever \cite{grushevsky-krichever} could construct a real-analytic foliation of $\M_{g,2}$ with holomorphic leaves and used it to reprove Diaz's bound.

The proof makes essential use of the holomorphicity of period coordinates and shares some similarities with the techniques employed in the present paper.
A more detailed account of the work of Diaz, Looijenga and Grushevsky-Krichever can be found in \cite{mondello:stratifications}.

\subsection{Strategy of the proof}

The basic idea to bound the Dolbeault cohomological dimension of a
complex manifold $X$
is to exploit Theorem \ref{thm:q-convex}, proven by Andreotti-Grauert \cite{andreotti-grauert}, which reduces the problem to exhibiting an exhaustion function $\exh$ (i.e. a proper real-valued function, bounded from below) on $X$ whose complex Hessian has controlled index.
More precisely, if the positivity index of $i\pa\ol{\pa}\exh$ is everywhere at least
$\dim_{\CC}(X)-q$, then the Dolbeault cohomology vanishes above degree $q$.

\begin{remark}
Approaching $\cohdim_{alg}$ of a complex algebraic manifold $X$
via exhaustion functions requires a careful study of the behavior at infinity: indeed, this is equivalent to considering a suitable projective compactification $\ol{X}$ of $X$ by a Cartier divisor $D=\ol{X}\setminus X$ at infinity and then looking
for a Hermitian metric $h$ on the line bundle $L=\Ocal_{\ol{X}}(D)$
whose curvature $-i\pa\ol{\pa}\log(h)$ has controlled index.
We will not discuss this approach any further in this paper.
\end{remark}

\subsubsection{Moduli of Abelian differentials}
Being unable to directly exhibiting a useful exhaustion function on $\M_{g,n}$, it turns out to be easier to first deal with a suitable projective bundle $\PP\Omega\M_{g,n}$ on $\M_{g,n}$ and then invoke (P5) to draw the wished conclusion for $\M_{g,n}$.
Such a $\PP\Omega\M_{g,n}$ is the projectivization of the moduli space $\Omega\M_{g,n}$ of non-zero Abelian differentials 
of type $(g,n)$:
a point of $\Omega\M_{g,n}$ represents a triple $(C,\vc{p},\ph)$,
where $C$ is a compact connected Riemann surface of genus $g$
with distinct marked points
$\vc{p}=(p_1,\dots,p_n)$ and $\ph$ is a non-zero Abelian differential (i.e. a holomorphic $(1,0)$-form) on $C$.
The map $\PP\Omega\M_{g,n}\rar\M_{g,n}$ that sends $(C,\vc{p},[\ph])$ to $(C,\vc{p})$ is clearly a holomorphic $\CC\PP^{g-1}$-bundle. Thus, our main result is actually equivalent to the following.

\begin{maintheorem}\label{mthm:hodge}
The Dolbeault cohomological dimension of the moduli space $\PP\Omega\M_{g,n}$ of projective Abelian differentials of type $(g,n)$ satisfies
\[
\cohdim_{Dol}(\PP\Omega\M_{g,n})\leq (2g-3+\eps_n)+g
\]
for all $g\geq 2$ and $n\geq 0$.
\end{maintheorem}

For technical reasons, it is more practical to work with the
moduli space $\PP\Omega\M'_{g,n}$
consisting of triples $(C,\vc{p},[\ph])$, where $C\in\M_g$,
$\vc{p}=(p_1,\dots,p_{n+2g-2})$ is a collection of points of $C$
such that $p_1,\dots,p_n$ are distinct, and $\ph$ is a non-zero Abelian differential on $C$ that vanishes on $p_{n+1},\dots,p_{n+2g-2}$.
The map $\PP\Omega\M'_{g,n}\rar\PP\Omega\M_{g,n}$ that forgets
the last $2g-2$ marked points is a finite surjective cover and so
the two spaces have the same Dolbeault cohomological dimension by (P4).

Unfortunately, we are unable to produce a satisfactory exhaustion function on $\PP\Omega\M'_{g,n}$, nor an approach via stratification using (P7) is available for $\cohdim_{Dol}$ (see Remark \ref{rmk:property7}).
Thus, in order to prove Theorem \ref{mthm:hodge}, 
we will use property (P6) of $\cohdim_{Dol}$ listed above; namely,
we will produce an open cover of $\PP\Omega\M'_{g,n}$ by suitably thickening
the strata described below
and we will show that each intersection of open sets in such a cover has a nice exhaustion function.

\subsubsection{Stratification}  
The moduli space $\PP\Omega\M'_{g,n}$ can be
stratified accordingly to the configuration of zeros of the Abelian differential,
or in other words accordingly to how the marked points collide: such information can be
kept track by a surjection $\sigma:\{1,2,\dots,n+2g-2\}\surj\{1,2,\dots,n+k\}$
that fixes $\{1,2,\dots,n\}$ pointwise.

Indeed, locally closed strata are the complex-algebraic loci
$\PP\Omega\M'_{g,n}(\sigma)$ inside $\PP\Omega\M'_{g,n}$
representing triples $(C,\vc{p},[\ph])$ such that
$p_i=p_{i'}\iff \sigma(i)=\sigma(i')$.

The hierarchy of such stratification is as expected: going to a deeper stratum corresponds to merging some marked points,
provided $p_1,\dots,p_n$ remain distinct.

\begin{remark}\label{rmk:zeroes}
Such a stratum $\PP\Omega\M'_{g,n}(\sigma)$ can be naturally identified
to the moduli space $\PP\Omega\M'_{g,n}(m_1,\dots,m_{n+k})$
of triples $(C,\vc{q},[\ph])$ such that $(C,\vc{q})\in\M_{g,n+k}$,
$[\ph]\in\PP H^{1,0}(C)$ and
%
$\ord_{q_j}(\ph)=m_j$,
where $m_j=|\sigma^{-1}(j)\cap\{n+1,\dots,n+2g-2\}|$ for $j=1,\dots,n+k$.
%
\end{remark}

As a warm-up and a basic step towards Theorem \ref{mthm:hodge}, we will first 
exhibit exhaustion functions with controlled complex Hessian on such strata,
thus bounding their Dolbeault cohomological dimension and obtaining a result
that may be of interest on its own.
%

\begin{maintheorem}\label{mthm:strata}
The Dolbeault cohomological dimension of the stratum $\PP\Omega\M'_{g,n}(\sigma)$ of
projective Abelian differentials satisfies
\[
\cohdim_{Dol}\left(\PP\Omega\M'_{g,n}(\sigma)\right)\leq g
\]
for all $g\geq 2$ and $n\geq 0$ and for all $\sigma$.
\end{maintheorem}

By Remark \ref{rmk:zeroes}, the above result can be rephrased in terms
of the moduli spaces $\PP\Omega\M'_{g,n}(m_1,\dots,m_{n+k})$.\\

We stress that Theorem \ref{mthm:strata} is already non-optimal in genus $2$, since $\PP\Omega\M'_2(2)$ and $\PP\Omega\M'_2(1,1)$ are affine, and in genus $3$, since 
$\PP\Omega\M'_{3}(4)$ and $\PP\Omega\M'_3(3,1)$
are affine
(see \cite{looijenga-mondello}).
As an extension of Looijenga's question, it seems natural to wonder the following.

\begin{question}
What is the Dolbeault/algebraic cohomological dimension of the strata $\PP\Omega\M'_{g,n}(\sigma)$?
\end{question}

\begin{question}
Is the deepest stratum in $\PP\Omega\M'_{g}$ affine for every $g\geq 2$?
\end{question}

%

\subsubsection{Period coordinates}
What makes computations comfortable inside each locally closed stratum is its smoothness and the existence of the so-called period coordinates. Very concretely,
the universal family over a contractible neighbourhood $U$ of
$(C,\vc{p},\ph)\in\Omega\M'_{g,n}(\sigma)$ can be topologically identified to $(C,\vc{p})\times U\rar U$ and so, for each $u\in U$, the corresponding Abelian differential $\ph_u$ on $C\times\{u\}$ determines a class 
$(\ph_u)\in H^1(C,\vc{p};\CC)$. Up to shrinking $U$, the induced map $U\rar H^1(C,\vc{p};\CC)$ is a biholomorphism onto its image.
Thus, a $\CC$-basis of $H_1(C,\vc{p};\CC)$ gives local period coordinates near $(C,\vc{p},\ph)$.

\subsubsection{Geometric functions on strata}
A non-zero Abelian differential $\ph$ on $C$ determines a non-positively curved metric $|\ph|^2$ on the Riemann surface with conical singularities at the zeros of $\ph$. Thus, there are a number of invariants that can be extracted from a point $(C,\vc{p},\ph)$
of $\Omega\M'_{g,n}(\sigma)$: for instance, the total area $A(\ph)$ of the metric $|\ph|^2$.
Once $\sigma$ and so the stratum are fixed, another such geometric function
is the systole $\ell_{sys}(\ph)$, namely the length (with respect to $|\ph|^2$) of the shortest
{\it{nontrivial}} geodesic path on $C$ joining two marked points.

The product $A\cdot\ell_{sys}^{-2}$ defines a positive real-valued function on $\PP\Omega\M'_{g,n}(\sigma)$ and it follows from Proposition 1 in \cite{kms:ergodicity} that $A\cdot\ell_{sys}^{-2}$ is an exhaustion function.
It is easy to realize that the (distributional) complex Hessian of $A\cdot\ell_{sys}^{-2}$ is not positive enough for our purposes, since it has at most $g$ positive eigenvalues. Thus, we will improve the situation as follows.

Given a collection $B$ of arcs joining marked points of $C$, we define
$\ell_B^{-2}(\ph):=\sum_{\gamma\in B}\ell^{-2}_\gamma(\ph)$,
where $\ell_\gamma(\ph)$ is the length of the $\ph$-shortest path homotopic to $\gamma$. 

Taking the sup of $\ell_B^{-2}$ over all collections $B$ of arcs that form a
$\RR$-basis of $H_1(C,\vc{p};\RR)$ defines a function
$\ell_\Bcal^{-2}:\Omega\M'_{g,n}(\sigma)\rar\RR$,
which is greater than $\ell_{sys}^{-2}$ and enjoys the same homogeneity property. Hence,
$\log(A\cdot\ell_\Bcal^{-2}):\PP\Omega\M'_{g,n}(\sigma)\rar\RR$ is again a well-defined exhaustion function: this is the function we wish to analyze.

\subsubsection{Complex Hessian of geometric functions}
Using period coordinates on the stratum $\Omega\M'_{g,n}(\sigma)$ it is easy to see that the 
complex Hessian $i\pa\ol{\pa}A$ at $(C,\vc{p},\ph)$ has $g$ negative directions corresponding to $H^{0,1}(C)$.
On the other hand, if the value of $\ell^{-2}_\Bcal(\ph)$ is attained
at a basis $B$, then every arc $\gamma\in B$ is realized by a smooth geodesic for $|\ph|^2$. Thus, $\ell^{-2}_\gamma(\ph)=|\int_\gamma \ph|^{-2}$ and so $i\pa\ol{\pa}\ell^{-2}_B$ is positive-definite in a neighbourhood of $\ph$.
As a consequence, $i\pa\ol{\pa}\ell^{-2}_\Bcal$ is positive-definite in the distributional sense.
A little computation shows that
the complex Hessian of $\log(A\cdot\ell^{-2}_\Bcal)$ on $\PP\Omega\M'_{g,n}(\sigma)$ has at worst $g$ non-positive directions, and this
proves Theorem \ref{mthm:strata}.

\begin{remark}
The fact that $i\pa\ol{\pa}A$ has index of negativity $g$ seems precisely responsible
for all the estimates being $g$ steps off the optimal conjectural ones.
\end{remark}

\subsubsection{Thickening the stratification}

The exhaustion function $\log(A\cdot\ell_\Bcal^{-2})$ on the stratum
$\PP\Omega\M'_{g,n}(\sigma)$
can be extended on a neighbourhood of its inside
$\PP\Omega\M'_{g,n}$.
Since the area functional $A$ is already defined on the whole $\Omega\M'_{g,n}$,
it is enough to define an extension of $\ell_\Bcal^{-2}$, which requires a little preparation.

Consider the universal family of Riemann surfaces over a small contractible neighbourhood $U\subset\Omega\M'_{g,n}$ of a point on the stratum $\Omega\M'_{g,n}(\sigma)$, which can be topological trivialized as $C\times U\rar U$.
We can view the data on $U$ as a family
$U\ni u\mapsto (J_u,\vc{p}(u),\ph_u)$, where
$J_u$ is a complex structure on $C$,
the $1$-form $\ph_u$ on $C$ is $J_u$-holomorphic and
and $p_1(u),\dots,p_{n+2g-2}(u)\in C$ are the marked points.
Being $u\in U$ close to the stratum $\Omega\M'_{g,n}(\sigma)$,
the marked points $p_{i}(u)$ are gathered in $n+k$ groups, contained in $n+k$ small disjoint disks $\disk^1(\ph_u),\dots,\disk^{n+k}(\ph_u)\subset C\times\{u\}$.
Thus, we can split the set of arcs joining marked points on $C\times\{u\}$ into the subset of ``inner'' ones, contained
inside $\disk(\ph_u)=\bigcup_{j=1}^{n+k} \disk^j(\ph_u)$, and the subset of ``outer'' ones, not contained in $\disk(\ph_u)$.
For example, see Figure \ref{fig:clash}.

We can define
$\ell^{-2}_{\Bcal^{out}}(\ph_u):=\mathrm{sup}\ \ell^{-2}_{B^{out}}(\ph_u)$, where $B^{out}$ ranges
over all bases of $H_1(C,\disk(\ph_u);\RR)$ made of $\ph_u$-outer segments.
Since $\ell^{-2}_{\Bcal^{out}}$ restricts to $\ell^{-2}_\Bcal$ on the stratum $\Omega\M'_{g,n}(\sigma)$,
the function $\eta=A\cdot\ell^{-2}_{\Bcal^{out}}$ is an extension of $A\cdot\ell^{-2}_\Bcal$.

On the same neighbourhood of $\Omega\M'_{g,n}(\sigma)$
we can also produce a function $\zeta$ whose complex Hessian is positive in the directions transverse to the stratum by letting $\zeta(\ph_u)=
\sup\ \zeta_{B^{inn}}(\ph_u)$, where $B^{inn}$ ranges over 
all bases of $H_1(\disk(\ph_u),\vc{p}(u);\RR)$ made of $\ph_u$-inner segments, and $\zeta_{B^{inn}}(\ph_u)=
\ell^{-2}_{\Bcal^{out}}(\ph_u)\sum_{\beta\in B^{inn}}|\int_\beta\ph_u|^2$.
See Figure \ref{fig:bases} for an example of inner and outer bases.

Now, the function $\eta+\zeta$ is proper ``along the stratum'', since $\eta$ is; moreover, it can
be easily shown that its complex Hessian has non-positivity index $\leq g$. However, $\eta+\zeta$ is not proper ``transversely to the stratum''.
This problem can be fixed by choosing a suitable convex 
and quickly diverging real function $\chi$
and using $\exh_\sigma=\log(\eta+\chi\circ\zeta)$,
which is an exhaustion function on a suitable thickening $W_\sigma$ of the stratum $\PP\Omega\M'_{g,n}(\sigma)$.
Moreover, by controlling the directions along which the $\exh_\sigma$'s 
are strictly subharmonic, we can also show that every non-empty intersection $W_{\sigma_0}\cap\dots\cap W_{\sigma_d}$ has $\cohdim_{Dol}\leq g$
by applying Theorem \ref{thm:q-convex} to the function $\exh_{\sigma_0}+\dots+\exh_{\sigma_d}$.

Finally, such thickenings can be carefully tailored 
so that 
$W_\sigma$ and $W_\tau$ intersect if and only if
$\PP\Omega\M'_{g,n}(\sigma)$ is contained in the closure of
$\PP\Omega\M'_{g,n}(\tau)$ or vice versa (see, for example, Figure \ref{fig:thickening}).
As a consequence, the nerve of the open cover $\Wfrak=\{W_\sigma\}$ 
of $\PP\Omega\M'_{g,n}$ has dimension $2g-3+\eps_n$, and Theorem \ref{mthm:hodge} follows from (P6) using a Mayer-Vietoris spectral sequence.\\

As a by-product of the above construction, we can also estimate $\cohdim_{Dol}$
of a closed stratum in $\PP\Omega\M'_{g,n}$.

\begin{maintheorem}\label{mthm:closed}
The Dolbeault cohomological dimension of the closure of
the stratum $\PP\Omega\M'_{g,n}(\tau)$ inside $\PP\Omega\M'_{g,n}$
is bounded above by $(k-1+\eps_n)+g$
%
for all $g\geq 2$, $n\geq 0$ and
for all surjections $\tau:\{1,\dots,2g-2+n\}\surj\{1,\dots,n+k\}$
that fix $\{1,2,\dots,n\}$.
\end{maintheorem}

\subsection{Acknowledgements}
I am grateful to Enrico Arbarello for many stimulating discussions
and for his constant encouragement and wise suggestions.
I would also like to thank Jean-Pierre Demailly
for very useful clarifications about $q$-convex functions and
vanishing theorems; Simone Diverio, always ready to
explain me things about complex geometry; and
Eduard Looijenga for helpful conversations about this topic.
Finally, I am thankful to the anonymous referee(s) for
carefully reading the paper and for pointing out meaningful
bibliographical references.

I am part of the INdAM national research group GNSAGA. My research has been partially supported by the MIUR grant FIRB 2010 ``Low-dimensional
geometry and topology'' ({RBFR10GHHH\_003}).

\section{Moduli spaces of Abelian differentials}\label{sec:moduli}

In this section we set some notation and review basic notions, such as moduli spaces
of Riemann surfaces, the Hodge bundle and its stratification by configurations of zeros,
the hierarchy of strata, smoothness and local period coordinates for the closed strata.

\subsection{Moduli of Riemann surfaces}

Let $g\geq 2$ and $n\geq 0$ and denote by $\M_{g,n}$ the
{\it{moduli space of Riemann surfaces of genus $g$
with $n$ distinct marked points}}
and by $\pi:\C_{g,n}\rar\M_{g,n}$ be the {\it{universal family}}.
A point in $\M_{g,n}$ will thus represent a couple $(C,\vc{p})$ up to isomorphism,
where $C$ is a compact connected Riemann
surface of genus $g$ and the injective map $\vc{p}:\num{n}=\{1,\dots,n\}\hra C$ is the marking.

\begin{remark}\label{remark:stack}
$\M_{g,n}$ has a natural structure of complex-analytic orbifold, and
even of Deligne-Mumford stack, which is a global quotient of a
smooth variety $\wti{\M}_{g,n}$ by a finite group $G$.
All constructions can be intended to be performed $G$-equivariantly on $\wti{\M}_{g,n}$.
Indeed, if $\rho:\wti{\M}_{g,n}\rar \M_{g,n}$ is the covering map and $E\rar\M_{g,n}$,
is a holomorphic vector bundle, 
then $\rho^*E\rar\wti{\M}_{g,n}$ is
a $G$-equivariant holomorphic vector bundle
and
$H^{0,q}_{\bar\pa}(\M_{g,n};E):=
H^{0,q}_{\bar\pa}(\wti{\M}_{g,n};\rho^*E)^{G}$
is well-defined and independent of the choice of the finite Galois cover.
Similar considerations hold for the cohomology of $\C_{g,n}$ and all the other moduli spaces
that appear in this article with coefficients in some analytic or algebraic coherent sheaf.
\end{remark}

\subsubsection{Allowing marked points to coalesce}

Here is a variant of the previous construction which is useful for technical purposes.
For every $k\geq 0$ we will denote by $\M_{g,n}^k$ the 
{\it{moduli space of Riemann surface of genus $g$ with $(n,k)$-marked points}} that parametrizes isomorphism classes of
couples $(C,\vc{p})$, where $C$
is a compact connected Riemann surfaces of genus $g$ 
and $\vc{p}:\num{n+k}\rar C$ is a map whose restriction to
$\num{n}\subseteq\num{n+k}$ is injective.

Forgetting the
last $k$ marked points defines a map $f_{\num{k}}:\M_{g,n}^k\rar\M_{g,n}$,
so that $\M_{g,n}^k$ can be identified to the $k$-th fiber product $\C_{g,n}\times_{\M_{g,n}}\times\dots
\times_{\M_{g,n}}\C_{g,n}$ and the map $\M_{g,n}^{k+1}\rar\M_{g,n}^k$ that forgets the last marked point can be identified
to the universal family $\pi_k:\C_{g,n}^k\rar\M_{g,n}^k$.
So, analogous considerations as in Remark~\ref{remark:stack} apply in this case.

Let $\Sfrak_n(k,l)$ the set of surjections $\sigma:\num{n+k}\surj\num{n+l}$ which restrict to the identity on $\num{n}$.
For each $\sigma\in\Sfrak_n(k,l)$, we obtain a map
$b_\sigma:\M_{g,n}^l\rar\M_{g,n}^k$ by letting $b_\sigma(C,\vc{q})=(C,\vc{q}\circ\sigma)$,
where $\vc{q}:\num{n+l}\rar C$. We denote by $\delta_\sigma\subset\M_{g,n}^k$ the closed image of $b_\sigma$.

If $\sigma'=\upsilon\circ\sigma$ for some $\upsilon\in\Sfrak_n(l,h)$, then 
$\delta_{\sigma'}\subseteq\delta_\sigma$ and
we will write $\sigma'\preceq\sigma$.
Notice that $\sigma'\preceq\sigma$ and $\sigma\preceq\sigma'$ simultaneously hold if and only
if $\sigma'=\upsilon\circ\sigma$ for some bijection $\upsilon\in\Sfrak_n(l,l)$: in this case we will write $\sigma\sim\sigma'$.


For $i<j$ and $j>n$, it will be occasionally useful to denote by $\delta_{i,j}$ the Cartier divisor of $\M_{g,n}^k$ on which $p_i=p_j$
and by $\delta$ the union of all such $\delta_{i,j}$.

\subsection{Stratification of the Hodge bundle}\label{subsection:stratification}

Consider
$\omega_{\pi_k}\rar\C_{g,n}^k$ the $\pi_k$-vertical holomorphic cotangent bundle.
The so-called {\it{Hodge bundle}} $(\pi_k)_*(\omega_{\pi_k})\rar \M^k_{g,n}$
is a holomorphic vector bundle of rank $g$.

\begin{definition}
The {\it{moduli space of Abelian differentials
$\Omega\M^k_{g,n}$ of type $(g,n,k)$}}
is the complement of the zero section inside $(\pi_k)_*(\omega_{\pi_k})$.
\end{definition}

A point of $\Omega\M^k_{g,n}$
represents a triple $(C,\vc{p},\ph)$ up to isomorphism, where $\ph$
is a non-zero holomorphic $(1,0)$-form on the $\vc{p}$-marked compact Riemann surface $C$ of genus $g$.



\begin{definition}
A {\it{multiplicity datum in genus $g$}} is an
$\vc{m}:\num{n+k}\rar\NN$ such that 
$m_i>0$ for all $i>n$ and
$\sum_{i=1}^k m_i=2g-2$. The
locally closed locus
\[
\Omega\M'_{g,n}(\vc{m}):=\big\{(C,\vc{p},\ph)\in\Omega\M_{g,n}^k\, \big|\,C\in\M_{g,n}^k,\
\ \ord_{p_i}(\ph)= m_i\quad \forall i\,\big\} 
\]
inside $\Omega\M_{g,n}^k$ is called
{\it{moduli space of Abelian differentials of type $(g,n,\vc{m})$}}.
Its closure inside $\Omega\M'_{g,n}$ is
\[
\Omega\M'_{g,n}(\ol{\vc{m}})
:=
\big\{(C,\vc{p},\ph)\in\Omega\M_{g,n}^k\, \big|\,C\in\M_{g,n}^k,\
\ \ord_{p_i}(\ph)\geq m_i\quad \forall i\,\big\} 
\]
and we
will call it the moduli space
of Abelian differentials of type $(g,n,\ol{\vc{m}})$.
\end{definition}

%
For brevity, the {\it{moduli space 
of Abelian differentials of type $(g,n)$ with marked zeros}}
$\Omega\M'_{g,n}(\ol{\vc{\mo}})$ corresponding to
$\vc{\mo}=(0^n,1^{2g-2})$ will be denoted
simply by $\Omega\M'_{g,n}$.
%

Notice that such a moduli space maps finitely over the moduli space $\Omega\M_{g,n}$ of Abelian differentials of type $(g,n)$
with {\it{unmarked}} zeros
as a $\Sfrak_{2g-2}$-branched cover.

\begin{definition}
The {\it{moduli space of projective Abelian differentials of type $(g,n,\vc{m})$}}
is the quotient $\PP\Omega\M'_{g,n}(\vc{m})$ of $\Omega\M'_{g,n}(\vc{m})$ by $\CC^*$ that acts by rescaling the Abelian differential.
Analogously we define $\PP\Omega\M_{g,n}^k$, $\PP\Omega\M'_{g,n}$ and $\PP\Omega\M'_{g,n}(\ol{\vc{m}})$. 
\end{definition} 

%
Notice that the natural projection identifies
$\PP\Omega\M'_{g,n}(\vc{m})$ with a subset of $\M_{g,n}^k$.
The image of $\Omega\M'_{g,n}(\vc{m})$ through the forgetful map
$\Omega\M_{g,n}^k\rar\Omega\M_{g,n}$ will be denoted by $\Omega\M_{g,n}(\vc{m})$, and similarly
$\PP\Omega\M_{g,n}(\vc{m})$ will be the corresponding locus inside $\PP\Omega\M_{g,n}$.
Thus, $\Omega\M_{g,n}(\vc{m})=\Omega\M_{g,n}(\vc{m}\circ\sigma)$ for every $\sigma\in\Sfrak_n(k,k)$.
The induced map $\PP\Omega\M'_{g,n}(\ol{\vc{m}})\rar \PP\Omega\M_{g,n}(\ol{\vc{m}})$ is a
Galois cover with group $\aut_n(\vc{m})=\{\lambda\in\Sfrak_n(k,k)\,|\,m_{\lambda(i)}=m_i\}$,
that branches away from $\PP\Omega\M_{g,n}(\vc{m})$.

\medskip

The loci $\PP\Omega\M'_{g,n}(\vc{m})$ and $\PP\Omega\M_{g,n}(\vc{m})$ 
are part of an algebraic stratification
of $\PP\Omega\M'_{g,n}$ 
and of $\PP\Omega\M_{g,n}$ respectively
that we want to describe.

By abuse of notation, we will denote by $\omega_i$
both the vertical cotangent line bundle $\omega_{i}\rar \M_{g,n}^k$
associated to the map $f_i$ 
the forgets the $i$-th marking
and its pull-back to $\PP\Omega\M_{g,n}^k$ or to any locus $\PP\Omega\M'_{g,n}(\vc{m})$.
We will write $J^r\omega_i$ for the $r$-th jet bundle (relative
to the family $f_i$) associated to $\omega_i$.
As usual $J^0\omega_i=\omega_i$ and there is a natural exact sequence
\[
0\rar  \omega_i^{\otimes (r+1)}\otimes\omega_i \rar J^{r+1}\omega_i \rar J^r\omega_i \rar 0
\]

\begin{notation}
We denote simply by $\Ocal(-1)$ the tautological line bundle
$\Ocal_{\PP\Omega\M_{g,n}^k}(-1)$ on $\PP\Omega\M_{g,n}^k$
and by $\Ocal(1)$ its dual.
For simplicity, we will use the same symbol for the restriction of
the line bundle to any locus inside $\PP\Omega\M_{g,n}^k$.
\end{notation}

Here we want to show the following.

\begin{lemma}[Chow classes of strata]
The moduli space ${\PP\Omega\M'_{g,n}(\ol{\vc{m}})}$ 
of Abelian differentials of type $(g,n,\ol{\vc{m}})$
is a local complete intersection inside $\PP\Omega\M_{g,n}^k$
and its class in the Chow ring $A^*(\PP\Omega\M_{g,n}^k)$ is
\[
\left[{\PP\Omega\M'_{g,n}(\ol{\vc{m}})}\right]=\prod_{i=1}^{n+k} 
\ \prod_{r=1}^{m_i}
\left(\h+r\, c_1(\omega_i)\right)
\]
where $\h=c_1(\Ocal(1))$.
The complement 
of $\PP\Omega\M'_{g,n}(\vc{m})$ inside its closure is
the Cartier divisor ${\PP\Omega\M'_{g,n}(\ol{\vc{m}})}\cap\delta$,
and so $\PP\Omega\M'_{g,n}(\vc{m})\subset\PP\Omega\M_{g,n}^k$ is locally closed.
\end{lemma}
\begin{proof}
The closed locus  $\PP\Omega\M'_{g,n}(\ol{\vc{m}})$ inside 
$\PP\Omega\M_{g,n}^k$ 
can be described as
\[
{\PP\Omega\M'_{g,n}(\ol{\vc{m}})}=\bigcap_{m_i>0} Z(\Ev_i^{m_i-1})=:Z(\Ev^{\vc{m}})
\]
where $Z$ denotes the zero locus
and $\Ev_i^{m_i-1}$ and $\Ev^{\vc{m}}$ are the evaluation maps
\begin{align*}
\Ev_i^{m_i-1} : \ & \Ocal(-1) \lra J^{m_i-1}\omega_i\\
\Ev^{\vc{m}}  : \ & \Ocal(-1) \lra
\bigoplus_{m_i>0} J^{m_i-1}\omega_i
\end{align*}

Thus we have described ${\PP\Omega\M'_{g,n}(\ol{\vc{m}})}$ as the zero locus
of a section of the holomorphic vector bundle
$\bigoplus_{m_i>0}
J^{m_i-1}\omega_i\otimes\Ocal(1)$
on $\PP\Omega\M_{g,n}^k$ of
rank $\sum_i m_i=2g-2$. As $\mathrm{dim}(\PP\Omega\M_{g,n}^k)=4g-4+n+k$,
we expect ${\PP\Omega\M'_{g,n}(\ol{\vc{m}})}$ to have pure dimension $2g-2+n+k$.
We will see in Section \ref{sec:infinitesimal} that this is indeed the case, and so we are done.
\end{proof}

Pushing down the above class of $\PP\Omega\M'_{g,0}(\vc{m})$ to $\M_{g,k}$ it is possible to recover a formula by
Chen (see \cite{chen:strata}, Proposition 2.3).
In a similar fashion, we can reobtain Proposition 3.1
of \cite{chen:strata} by
pushing down the class of $\PP\Omega\M'_{g,0}(2,1^{2g-4})$
to $\PP\Omega\M_{g,0}$; in the same paper, the author
also computes
the class of the closure of $\PP\Omega\M'_{g,0}(2,1^{2g-4})$
inside the Deligne-Mumford compactification of $\PP\Omega\M_{g,0}$.
%
%

\subsection{Hierarchy of strata}

Let $k$ and $\vc{m}$ be as above.
For every $\sigma\in\Sfrak_n(k,l)$,
there is an identification
\[
{\PP\Omega\M'_{g,n}(\ol{\sigma_*\vc{m}})}\arr{\sim}{\lra} 
{\PP\Omega\M'_{g,n}(\ol{\vc{m}})}\cap\delta_\sigma
\]
given by $(C,\vc{q},[\ph])\mapsto (C,\vc{q}\circ\sigma,[\ph])$, where
$\sigma_* \vc{m}:\num{n+l}\rar\NN$ is defined as
$(\sigma_* \vc{m})_j:=\sum_{\sigma(i)=j}m_i$.

The closed substratum $\PP\Omega\M'_{g,n}(\ol{\vc{m}})\cap\delta_{\sigma}$
coincides with $\PP\Omega\M'_{g,n}(\ol{\vc{m}})\cap\delta_{\sigma'}$ if and only if $\sigma\sim\sigma'$.

\begin{notation}
The set 
of such equivalence classes can be identified
to the subset $\Sfrak_n^{ord}(k,l)\subset\Sfrak_n(k,l)$ of {\it{ordered surjections}}, i.e.
of $\sigma\in\Sfrak_n(k,l)$ such that $\min \sigma^{-1}(j)<\min\sigma^{-1}(j')\iff j<j'$.
\end{notation}

As a consequence
\[
{\PP\Omega\M'_{g,n}(\ol{\vc{m}})}
=\bigsqcup_{d=0}^{k-1+\eps_n} \PP\Omega\M'_{g,n}(\vc{m})^d,
\quad
\text{where}
\ \eps_n=
\begin{cases}
0 & \text{if $n=0$}\\
1 & \text{if $n>0$}
\end{cases}
\]
where
\[
\PP\Omega\M'_{g,n}(\vc{m})^d \ \cong
\bigsqcup_{\sigma\in\Sfrak^{ord}_n(k,k-d)} \PP\Omega\M'_{g,n}(\sigma_*\vc{m})\,
\]
is the locally closed stratum of codimension $d$.
\begin{notation}
As $n$ is understood, we will use the simplified symbols
$\Sfrak$ for $\bigcup_{d=0}^{2g-3+\eps_n} \Sfrak^{ord}_n(2g-2,2g-2-d)$
and we will also say that the surjection
$\sigma\in\Sfrak$ has {\it{depth $d(\sigma)=d$}} if $\sigma\in\Sfrak_n(2g-2,2g-2-d)$.
Moreover, 
for every $\sigma\in\Sfrak$
we will write $\PP\Omega\M'_{g,n}(\sigma)$
for $\PP\Omega\M'_{g,n}(\sigma_*\vc{\mo})$, where $\vc{\mo}=(0^n,1^{2g-2})$ corresponds
to the unique open stratum $\PP\Omega\M'_{g,n}(\vc{\mo})$ inside $\PP\Omega\M'_{g,n}$.
\end{notation}

In particular, the whole moduli space of Abelian differentials of type $(g,n)$ with
marked zeros is stratified as
\[
\PP\Omega\M'_{g,n} \cong
\bigsqcup_{\sigma\in\Sfrak} \PP\Omega\M'_{g,n}(\sigma)
\, .
\]
Quite similarly, for the moduli space of Abelian differentials with unmarked zeros
we have $\PP\Omega\M_{g,n}=\bigsqcup_{d=0}^{2g-3+\eps_n} (\PP\Omega\M_{g,n})^d$,
where
\[
(\PP\Omega\M_{g,n})^d\ =
\bigsqcup_{
\substack{m_1,\dots,m_n \\
m_{n+1}\geq\dots\geq m_{n+2g-2-d}}} \PP\Omega\M_{g,n}(m_1,\dots,m_{n+2g-2-d})\, .
\]
As a consequence of the above description, 
we have the following.

\begin{lemma}[Closed substrata are complete intersections]
The moduli space $\PP\Omega\M'_{g,n}(\ol{\sigma_*\vc{m}})$ is a complete intersection inside
${\PP\Omega\M'_{g,n}(\ol{\vc{m}})}$. In particular, the complement of $\PP\Omega\M'_{g,n}(\sigma_*\vc{m})$
inside its closure is a Cartier divisor and so
the complement of $\PP\Omega\M_{g,n}(\sigma_*\vc{m})$ 
inside ${\PP\Omega\M_{g,n}(\ol{\sigma_*\vc{m}})}$
is $\QQ$-Cartier.
\end{lemma}
\begin{proof}
The first claim follows from the fact that $\delta_{i,j}$ is Cartier and so $\delta_\sigma$ is a complete intersection. The complement
of $\PP\Omega\M'_{g,n}(\sigma_*\vc{m})$ in its closure is a union of smaller strata and so
it is a union of Cartier divisors, and so Cartier.
Finally,
the forgetful map ${\PP\Omega\M'_{g,n}(\ol{\sigma_*\vc{m}})}\rar{\PP\Omega\M_{g,n}(\ol{\vc{m}})}$ is a finite
(ramified) Galois cover that preserves the locus where some zeros coalesce: this implies the third claim. 
\end{proof}
 
In terms of evaluation maps and of the classes $\h$ and $c_1(\omega_i)$, we have the following.

\begin{lemma}[Degeneration divisor of a stratum]
The locus $\delta
\cap{\PP\Omega\M'_{g,n}(\ol{\vc{m}})}$
where the multiplicity datum of zeros of the Abelian differential
is more degenerate than $\vc{m}$
is the support of the effective divisor
\[
(n+k)\h+\sum_{i=1}^{n+k} (m_i+1)c_1(\omega_i)=
\sum_{\substack{j>n\\ i<j}} (m_i+m_j)[\delta_{i,j}] 
\quad\in A^1\left({\PP\Omega\M'_{g,n}(\ol{\vc{m}})}\right)
\]
\end{lemma}

\begin{proof}
Consider the evaluation map
\[
\ev_i: \Ocal(-1) \lra \omega_i^{\otimes(m_i+1)}
\]
defined on ${\PP\Omega\M'_{g,n}(\ol{\vc{m}})}$, which vanishes exactly $m_j$ times
on $\delta_{i,j}$ for all $j\neq i$. Thus
\[
\h+(m_i+1)c_1(\omega_i)=[Z(\ev_i)]=\sum_{j\neq i}m_j[\delta_{i,j}]
\quad\in A^1\left({\PP\Omega\M'_{g,n}(\ol{\vc{m}})}\right)
\]
and we conclude summing up over $i=1,\dots,n+k$
and remembering that $\delta_{i,j}=\emptyset$ if $i<j\leq n$.
\end{proof}

By pushing the class of $\delta
\cap{\PP\Omega\M'_{g,0}(\ol{\vc{m}})}$ down to $\PP\Omega\M_g$,
we can recover Korotkin-Zograf's relation from Lemma 4 of \cite{korotkin-zograf:tau}. In fact, they are also able to extend such a relation to the Deligne-Mumford compactification
for $\PP\Omega\M'_{g,0}$.

\subsection{Local properties of strata}\label{sec:infinitesimal}


Smoothness of the locally closed strata
are already in Veech \cite{veech:flat}, Polishchuk \cite{polishchuk:effective} and M\"oller \cite{moeller:linear}, and the existence of local period coordinates is shown in
\cite{veech:flat} and \cite{moeller:linear}.
Both results could be extracted from Hubbard-Masur \cite{hubbard-masur:quadratic}, where they work with quadratic differentials.

\medskip

Let
$\vc{m}=(m_1,\dots,m_{n+k})$ be a multiplicity datum in genus $g\geq 2$.
%

\begin{proposition}[Closed strata are smooth]\label{prop:smooth-closure}
The moduli space ${\Omega\M'_{g,n}(\ol{\vc{m}})}$ 
of Abelian differentials of type $(g,n,\ol{\vc{m}})$
is smooth of pure dimension $2g-1+n+k$.
\end{proposition}

\begin{proof}[Sketch of a proof]
Minor modifications of the arguments in
Polishchuk \cite{polishchuk:effective} and M\"oller \cite{moeller:linear}
show that $T_{(C,\vc{p},\ph)}\Omega\M'_{g,n}(\ol{\vc{m}})$ can
be identified to the hypercohomology group $\Hbb^1(C,\Dcal)$,
where $\Dcal=\Tcal\oplus\Kcal[-1]$ is the (shifted) cone of the map
\[
\Tcal=\left[
T_C \rar \bigoplus_{i=1}^{n+k} T_C|_{p_i} \right]
\lra
\Kcal=\left[
K_C \rar \bigoplus_{i=1}^{n+k} J^{m_i-1}_{p_i}K_C
\right]
\]
of complexes in degrees $[0,1]$ induced by $T_C \ni v \mapsto d(v\cdot\ph)\in K_C$,
and $J^{m_i-1}_{p_i}K_C$ is the $m_i$-dimensional vector space of $(m_i-1)$-jets of $K_C$ at the point $p_i$.

The conclusion follows from the long exact sequence in hypercohomology associated to the triangle $\Kcal[-1]\rar\Dcal\rar \Tcal$,
by observing that $\Hbb^0(C,\Dcal)=0$
and by showing that the cokernel of the map $\Hbb^1(C,\Tcal)\rar \Hbb^2(C,\Kcal[-1])=\Hbb^1(C,\Kcal)$
is given by the surjection $\Hbb^1(C,\Kcal)\rar H^2(C,\CC)\cong\CC$.
\end{proof}

In a similar way, one can also prove the following.

\begin{proposition}[Closed substrata are submanifolds]
For every surjection $\sigma\in\Sfrak_n(k,l)$,
the closed stratum $\PP\Omega\M'_{g,n}(\ol{\vc{m}})\cap\delta_\sigma$ is a regular submanifold of the moduli space ${\PP\Omega\M'_{g,n}(\ol{\vc{m}})}$ of Abelian differentials of type $(g,n,\ol{\vc{m}})$.
\end{proposition}

In order to recall the definition of period coordinates,
we need first to determine the correct period domain
for the moduli space $\Omega\M'_{g,n}(\ol{\vc{m}})$.
Let $\pi:\Omega\C^k_{g,n}\rar\Omega\M^k_{g,n}$ be the universal family and let
$\vc{P}=(P_1,\dots,P_{n+k})$ be the collection of tautological sections
$P_i:\Omega\M^k_{g,n}\rar\Omega\C^k_{g,n}$
corresponding to the marked points. 

\begin{notation}
We will use the symbol $\ul{\CC}_\C$ for the trivial $\CC$-vector bundle of rank $1$ on $\Omega\C^k_{g,n}$ and
by $\ul{\CC}_{P_i}$ its restriction to the $i$-th tautological section.
We also let
\[
\ul{\CC}_{\C,\vc{P}}:=
\left[
\ul{\CC}_\C \arr{ev}{\lra} \bigoplus_{i=1}^{n+k} \ul{\CC}_{P_i}
\right]
\]
be the complex of constructible sheaves in degrees $[0,1]$.
\end{notation}

The flat vector bundle $R^1\pi_{k,*}(\ul{\CC}_{\C,\vc{P}})$
on $\M^k_{g,n}$ has constant rank $2g+n+k-1$.
Moreover, its fiber over $(C,\vc{p})$ 
naturally contains $H^1(C,\vc{p};\CC)$ and they
actually agree if and only if $(C,\vc{p})\notin \delta$.
%

\begin{notation}
We will use the same symbol $\Hloc^1(\ul{\CC}_{\C,\vc{P}})$
to denote the flat vector bundle $R^1\pi_{k,*}(\ul{\CC}_{\C,\vc{P}})$, its restriction to any locus in $\M_{g,n}^k$ (such as
$\PP\Omega\M'_{g,n}(\ol{\vc{m}})$) and its pull-back to
$\Omega\M'_{g,n}(\ol{\vc{m}})$.
\end{notation}

%
%

In order to locally define the period map,
we pick a simply connected
open subset $U$ of ${\Omega\M'_{g,n}(\ol{\vc{m}})}$
and
let $u_0$ be a point in $U$ 
representing a triple $(C,\vc{p}(0),\ph_0)$.
We can smoothly (but not holomorphically!) trivialize the {\it{unmarked}} universal family $\Omega\C'_{g,n}(\ol{\vc{m}})|_U\cong C\times U\rar U$, which we can thus think of
as a map $u\mapsto (J_u,\vc{p}(u),\ph_u)$, where
$\vc{p}(u)$ denotes the subset of marked points
and $\ph_u$ a closed $1$-form
on the fixed surface $C$ at $u\in U$. Also, 
we fix a trivialization of the
flat vector bundle $\Hloc^1(\ul{\CC}_{\C,\vc{P}})$
over $U$ that identifies it to
${\Hbb^1(\ul{\CC}_{C,\vc{p}(0)})}\times U$.

For every $u\in U$, define $\Pcal(J_u,\vc{p}(u),\ph_u)
\in\mathrm{Hom}(H_1(C,\vc{p}(u);\ZZ),\CC)\cong H^1(C,\vc{p}(u);\CC)$ 
as
\[
\hspace{-0.5cm}\Pcal(J_u,\vc{p}(u),\ph_u):
\xymatrix@R=0in{
H_1(C,\vc{p}(u);\ZZ) \ar[rr] && \CC\\
\gamma \ar@{|->}[rr] && \int_\gamma \ph
}
\]
Remembering that $H^1(C,\vc{p}(u);\CC)\subseteq \Hbb^1(\ul{\CC}_{C,\vc{p}(u)})
\cong \Hbb^1(\ul{\CC}_{C,\vc{p}(0)})$, this construction 
induces a {\it{local period map}} $\Pcal:U\rar \Hbb^1(\ul{\CC}_{C,\vc{p}(0)})$. 
All such local period maps assemble into a section
$\Pcal$ of the flat vector bundle $\Hloc^1(\ul{\CC}_{\C,\vc{P}})$ over
$\Omega\M'_{g,n}(\ol{\vc{m}})$: the {\it{global period map}}.

\begin{proposition}[Period map]\label{prop:period-closed-hol}
The section $\Pcal$ of $\Hloc^1(\ul{\CC}_{\C,\vc{P}})$ is holomorphic.
Moreover, given a trivialization of the flat vector bundle $\Hloc^1(\ul{\CC}_{\C,\vc{P}})\cong
\Hbb^1(\ul{\CC}_{C,\vc{p}(0)})\times U$ over an open subset $U\subset\Omega\M'_{g,n}(\ol{\vc{m}})$,
the induced period map $\Pcal:U\rar \Hbb^1(\ul{\CC}_{C,\vc{p}(0)})$ is
a local biholomorphism.
\end{proposition}
\begin{proof}[Sketch of a proof]
It is enough to show that the differential $d\Pcal$ of the local period map is
$\CC$-linear and invertible
at each point $u\in U$.

Since $\ul{\CC}_{C,\vc{p}(u)}$ is quasi-isomorphic to the cone
$\widehat{\Dcal}:=\Ocal^\bullet\oplus\Kcal[-1]$ of the morphism
\[
\Ocal^\bullet=\left[
\Ocal_C \rar \bigoplus_{i=1}^{n+k} J^{m_i}_{p_i}\Ocal_C \right]
\arr{d}{\lra}
\Kcal=\left[
K_C \rar \bigoplus_{i=1}^{n+k} J^{m_i-1}_{p_i}K_C
\right]
\]
of complexes in degrees $[0,1]$ on $(C,J_u)$,
the argument essentially reduced to noticing that the morphism
$\Dcal\rar \widehat{\Dcal}$
induced by the multiplication map $T_C\arr{\cdot\ph_u}{\lra}\Ocal_C$
is a quasi-isomorphism,
where $\Dcal$ is the complex of coherent sheaves
associated to $(C,J_u,\vc{p}(u),\ph_u)$
that appears in the proof of Proposition \ref{prop:smooth-closure}.
Details can be found for instance in \cite{moeller:linear}.

See also \cite{veech:flat} for an equivalent approach.
\end{proof}

Since period coordinates transverse to substrata will be relevant
in studying their thickenings,
we want to analyze them more closely.

Let $\sigma\in\Sfrak_n(k,l)$ be a surjection for some $l<k$ and consider the identification of 
$\Omega\M'_{g,n}(\ol{\sigma_*\vc{m}})$ with
$\Omega\M'_{g,n}(\ol{\vc{m}})\cap\delta_\sigma$.
Fix a point $(C,\vc{q},\ph)$ in the former moduli space,
which thus corresponds to $(C,\vc{p},\ph)$ in the latter,
where $\vc{p}=\vc{q}\circ\sigma:\num{n+k}\rar C$.
%

The following is an easy consequence of Proposition \ref{prop:period-closed-hol}.

\begin{proposition}[Transverse period coordinates]
The period maps $\Pcal_{\ol{\vc{m}}}$ and $\Pcal_{\ol{\sigma_*\vc{m}}}$ 
associated to the strata $\Omega\M'_{g,n}(\ol{\vc{m}})$ and $\Omega\M'_{g,n}(\ol{\sigma_*\vc{m}})$
induce
an isomorphism of the normal bundle of ${\Omega\M'_{g,n}(\ol{\sigma_*\vc{m}})}$
inside ${\Omega\M'_{g,n}(\ol{\vc{m}})}$ with the sheaf
$\Hloc^1(\ul{\CC}_{\vc{Q},\vc{P}})$, where
$\ul{\CC}_{\vc{Q},\vc{P}}=\left[ \bigoplus_j \ul{\CC}_{Q_j}\arr{ev}{\lra}
\bigoplus_i \ul{\CC}_{P_i}\right]$, the $Q_j$ (resp. the $P_i$)
are the tautological sections of $\Omega\M'_{g,n}(\ol{\sigma_*\vc{m}})$ (resp. of $\Omega\M'_{g,n}(\ol{\vc{m}})$) and $ev(1_{Q_j})=\sum_{\sigma(i)=j}1_{P_i}$.
\end{proposition}

A treatment for coordinates normal to strata of Abelian differentials with no marked points as symmetric functions in the transverse periods can be found in
Kontsevich-Zorich \cite{kontsevich-zorich}.

\section{Length and area functionals}\label{sec:geometric}

Let $g\geq 2$ 
and let $\vc{m}:\num{n+k}\rar\NN$ 
be a multiplicity datum in genus $g$.

The purpose of this section is to study
certain real-valued homogeneous functions on
the moduli spaces $\Omega\M'_{g,n}(\vc{m})$ 
defined in terms of the flat metric
induced by the Abelian differential, such as the total area
of the flat surface or the flat length of the shortest geodesic
in a fixed homotopy class.
The complex Hessian of such functions will then be easily analyzed using period coordinate.

Then we will combine such functions to produce one that descends to $\PP\Omega\M'_{g,n}(\vc{m})$
and which is $\aut_n(\vc{m})$-invariant.

In taking
products and sums of such functions,
we need to keep control of the directions along which
their complex Hessian is positive-definite.
For this reason, we will often refer to the technical condition
described in Section \ref{sec:star}.


\subsection{Functions satisfying condition ${(\star)}$}\label{sec:star}

Consider a function $\eta:\cdom \lra \RR$
defined on an open cone $\cdom\subseteq \Omega\M'_{g,n}(\vc{m})$,
that is an open subset $\cdom$ that satisfies $\CC^*\cdot\cdom=\cdom$,
and let $(C,\vc{p},\ph)\in\cdom$.
We will usually deal with functions $\eta$ that are {\it{homogeneous of degree $d$}},
namely that
satisfy $\eta(C,\vc{p},\lambda\ph)=|\lambda|^d\,\eta(C,\vc{p},\ph)$ for all $\lambda\in\CC^*$
and all $(C,\vc{p},\ph)\in\cdom$.

Because of Proposition \ref{prop:period-closed-hol},
there exists a contractible open neighbourhood $U\subset \cdom$ of
$(C,\vc{p},\ph)$ such that the 
local period map $\Pcal:U\rar H^1(C,\vc{p};\CC)$
is biholomorphic onto an open subset $\Pcal(U)\subset H^1(C,\vc{p};\CC)$.

\begin{notation}
The Hodge decomposition of $H^1(C;\CC)$ with respect to the
complex structure of $(C,\vc{p},\ph)$ will be denoted by
$H^{1,0}_\ph(C)\oplus H^{0,1}_\ph(C)$.
\end{notation}

Call $\Pi_\ph\subset H^1(C,\vc{p};\CC)$
the subset of differential forms whose
absolute cohomology class belongs to $H^{1,0}_\ph(C)$, so that
\[
\xymatrix@R=0.3cm{
0 \ar[r] & \wti{H}^0(\vc{p};\CC) \ar[r] &
H^1(C,\vc{p};\CC) \ar[r] & H^1(C;\CC) \ar[r] & 0\\
0 \ar[r] & \wti{H}^0(\vc{p};\CC) \ar[r] \ar@{=}[u]  & 
\Pi_\ph \ar[r] \ar@{^(->}[u] & H^{1,0}_\ph(C) \ar@{^(->}[u] \ar[r] & 0
}
\]
commutes, where $\wti{H}^0(\vc{p};\CC)=H^0(\vc{p};\CC)/\CC$.
Let $\Pi'_\ph$ be a complement of the line $\CC(\ph)$ inside $\Pi_\ph$.

\begin{definition}\label{def:star}
A function $\eta$ {\it{weakly}} satisfies {\it{condition ${(\star)}$
at ${(C,\vc{p},\ph)}$}} if, up to taking a smaller $U$,
the restriction of $\eta$ to $U\cap \Pcal^{-1}((\ph)+\Pi_\ph)$ is plurisubharmonic.
It satisfies {\it{condition ${(\star)}$}}
if it weakly satisfies $(\star)$ and moreover
the restriction to $U\cap \Pcal^{-1}((\ph)+\Pi'_\ph)$ is strictly plurisubharmonic.
We will say that $\eta$ (weakly) satisfies condition $(\star)$
in $\cdom$ if it does at every point of $\cdom$.
\end{definition}

\begin{remark}
If $\eta$ is $C^2$, then it weakly satisfies
$(\star)$ at $(C,\vc{p},\ph)$ if the restriction of $(i\pa\ol{\pa}\eta)_{(C,\vc{p},\ph)}$
to the subspace $d\Pcal_{(C,\vc{p},\ph)}^{-1}(\Pi_\ph)$ is semipositive-definite; and it satisfies $(\star)$ at $(C,\vc{p},\ph)$ if
moreover the restriction of $(i\pa\ol{\pa}\eta)_{(C,\vc{p},\ph)}$
to the subspace $d\Pcal_{(C,\vc{p},\ph)}^{-1}(\Pi'_\ph)$ is positive-definite.
\end{remark}

We collect some obvious properties in the following statement.

\begin{lemma}[Condition ($\star$)]\label{lemma:star-properties}
The following properties hold.
\begin{itemize}
\item[(a)]
If $\eta$ satisfies weak $(\star)$, then $\lambda\eta$ satisfies
weak $(\star)$ for all $\lambda\geq 0$.
\item[(b)]
If $\zeta$ satisfies $(\star)$, then $\nu\zeta$ satisfies $(\star)$
for all $\nu>0$.
\item[(c)]
If $\eta_1,\eta_2$ satisfy weak $(\star)$ at $(C,\vc{p},\ph)$, then
$\eta_1+\eta_2$ satisfies weak $(\star)$ at $(C,\vc{p},\ph)$.
\item[(d)]
If $\zeta$ satisfies $(\star)$ and $\eta$ satisfies weak $(\star)$ in $\cdom$, then $\zeta+\eta$ satisfies $(\star)$ in $\cdom$.
\item[(e)]
If $\eta_1$ and $\eta_2$ are positive functions such that
$\log(\eta_1)$ and $\log(\eta_2)$ satisfy weak $(\star)$,
then $\log(\eta_1+\eta_2)$ satisfies weak $(\star)$.
Suppose moreover that $\eta_1,\eta_2\in C^2$ and,
for every $(C,\vc{p},\ph)\in\Xi$ and 
along every direction of $d\Pcal^{-1}_{(C,\vc{p},\ph)}(\Pi'_\ph)$,
either $i\pa\ol{\pa}\log(\eta_1)>0$ or
$i\pa\ol{\pa}\log(\eta_2)>0$.
%
%
Then $\log(\eta_1+\eta_2)$ satisfies $(\star)$.
\end{itemize}
\end{lemma}
\begin{proof}
Properties (a), (b) and (c) are obvious.

As for (d), let $U\subset\cdom$ be an open neighbourhood
of $(C,P,\ph)$ as above
and let $L:=U\cap \Pcal^{-1}((\ph)+\Pi'_\ph)$.
Notice that, along the segment $U\cap \Pcal^{-1}(\CC(\ph))$
both $\zeta$ and $\eta$ are plurisubharmonic and so $\zeta+\eta$ is.
Because the restriction of $\zeta$ to $L$ is strictly plurisubharmonic,
there exists a smooth $f: L\rar\RR$
such that $\zeta|_{L}-f$ is plurisubharmonic and
$(i\pa\ol{\pa}f)>0$ on $L$.
Hence, $f+\eta|_{L}$ is strictly plurisubharmonic and so
$(\zeta+\eta)|_{L}$ is too.

As for (e),
restrict to a germ of complex curve through $(C,\vc{p},\ph)\in \Xi$. Then by direct computation
\[
\hspace{-0.5cm}
i\pa\ol{\pa}\log(\eta_1+\eta_2)=
i\frac{\eta_1 \eta_2
\left(
\pa\log(\eta_1/\eta_2)\wedge\ol{\pa}\log(\eta_1/\eta_2)
\right)
+(\eta_1+\eta_2)\left(
\eta_1\pa\ol{\pa}\log(\eta_1)+\eta_2\pa\ol{\pa}\log(\eta_2)
\right)
}
{(\eta_1+\eta_2)^2}
\]
and the conclusion follows.
\end{proof}

\begin{notation}
A function $\PP\Omega\M'_{g,n}(\vc{m})\rar\RR$ corresponds to a $\CC^*$-invariant function $\Omega\M'_{g,n}(\vc{m})\rar\RR$. We will say that the former (weakly) satisfies condition ($\star$) if the latter does. Similarly,
we say that $\Omega\M_{g,n}(\vc{m})\rar\RR$ (weakly) satisfies
condition $(\star)$ if the induced $\aut_n(\vc{m})$-invariant 
$\Omega\M'_{g,n}(\vc{m})\rar\RR$ does.
\end{notation}

\begin{remark}
The definitions and remarks of the current section can be also given for real-valued functions
defined on an open subset of ${\PP\Omega\M'_{g,n}(\ol{\vc{m}})}$.
Indeed, in this case the local period map lands in $\Hbb^1(\ul{\CC}_{C,\vc{p}})$;
so we still have an exact sequence
\[
\xymatrix{
0 \ar[r] & \Hbb^0(\ul{\CC}_{\vc{p}})/\CC \ar[r] & \Hbb^1(\ul{\CC}_{C,\vc{p}}) \ar[r] &
H^1(C;\CC) \ar[r] & 0
}
\]
and we will denote by $\Pi_\ph$ the subspace of $\Hbb^1(\ul{\CC}_{C,\vc{p}})$
that maps to $H^{1,0}_{\ph}(C)$.
\end{remark}

\subsection{The area functional}

Let $\pi:\C_{g}\rar\M_{g}$ be the universal family and
let $\Hloc^1(\ul{\CC}_\C):=R^1\pi_*\ul{\CC}$ be the flat vector bundle
over $\M_{g}$, whose fiber over $C\in\M_{g}$ is $H^1(C;\CC)$.
We define a Hermitian pairing $h$ on $\Hloc^1(\ul{\CC}_\C)$ as
\[
h_{C}(\ph_1,\ph_2):=\frac{i}{2}\int_C \ph_1\wedge\ol{\ph}_2\, .
\]

\begin{lemma}[Standard Hermitian pairing]\label{lemma:b}
The Hermitian form $h$ has the following properties.
\begin{itemize}
\item[(a)]
For every $C\in\M_{g}$, the pairing
$h_{C}$ has signature $(n_+,n_-,n_0)=(g,g,0)$.
\item[(b)]
The restriction of $h_{C}$ to $H^{1,0}(C)$ is positive-definite
and its restriction to $H^{0,1}(C)$ is negative-definite, where
$H^1(C;\CC)=H^{1,0}(C)\oplus H^{0,1}(C)$ is Hodge decomposition
determined by $C\in\M_{g}$.
\item[(c)]
$h$ is smooth and it restricts to a smooth
Hermitian metric on the holomorphic vector bundle
$\pi_*(\omega_\pi)\rar\M_{g}$.
\end{itemize}
\end{lemma}
\begin{proof}
About (b),
a holomorphic $(1,0)$-form $\ph$ on $C$ can be locally written as
$f(z)dz$ with respect to a holomorphic coordinate $z=x+iy$.
So $\frac{i}{2}\ph\wedge\ol{\ph}=|f(z)|^2 dx\wedge dy$,
which implies that $h|_{H^{1,0}}$ is positive-definite.
An analogous argument shows that $h|_{H^{0,1}}$ is negative-definite.
Clearly, (a) follows from (b).

Finally, $h$ is smooth because it is locally constant
and (c) then follows from (b), because 
$\pi_*(\omega_\pi)$ is the subbundle of $\Hloc^1(\ul{\CC}_\C)$
consisting of $(1,0)$-holomorphic forms.
\end{proof}

\begin{definition}
The {\it{area functional}} $A:\Omega\M_g\rar \RR$
is
\[
A(C,\ph):=h_C(\ph,\ph)=\frac{i}{2}\int_C \ph\wedge\ol{\ph}\, .
\]
\end{definition}

Notice that indeed $A$ computes the area of the surface $C$
endowed with the flat metric $|\ph|^2$ with conical singularities.

\begin{notation}
By a little abuse, we will use the same symbol $A$
for the pull-back of the area functional to $\Omega\M_{g,n}^k$ and
for its restrictions to $\Omega\M'_{g,n}(\ol{\vc{m}})$ and $\Omega\M_{g,n}(\ol{\vc{m}})$.
\end{notation}

We collect in the following statement certain local features of the area functional.

\begin{lemma}[Area functional]\label{lemma:A}
The function $A:\Omega\M'_{g,n}(\vc{m})\rar\RR_{+}$
has the following properties:
\begin{itemize}
\item[(a)]
$A$ and $\log(A)$ are smooth;
\item[(b)]
$A$ is homogeneous of degree $2$;
\item[(c)]
the signature of $i\pa\ol{\pa}A$ is $(g,g,k-1)$ and
the signature of $i\pa\ol{\pa}\log(A)$ is $(g-1,g,k)$;
\item[(d)]
both $A$ and $\log(A)$ satisfy condition weak $(\star)$.
\end{itemize}
\end{lemma}
\begin{proof}
Part (a) follows from Lemma \ref{lemma:b}(c)
and part (b) is obvious.

Since claims (c) and (d) are local,
we can fix a $(C,\vc{p},\ph)$ in $\Omega\M'_{g,n}(\vc{m})$
and 
work in period coordinated near such point.
More precisely,
we pick a small open neighbourhood
$U\subset\Omega\M'_{g,n}(\vc{m})$ of $(C,\vc{p},\ph)$ such that the period map
realizes a biholomorphism between $U$ and an open subset
$\Pcal(U)\subset H^1(C,\vc{p};\CC)$,
whose existence is granted by Proposition \ref{prop:period-closed-hol}. 

Consider then the exact sequence
\[
0 \rar \wti{H}^0(\vc{p};\CC) \rar H^1(C,\vc{p};\CC) \rar H^1(C;\CC)\rar 0
\]
and notice that $A$ factorizes through $H^1(C;\CC)$,
so that $\wti{H}^0(\vc{p};\CC)$, which has dimension $k-1$,
certainly belongs to the radical of $i\pa\ol{\pa}A$.
More explicitly, if $\dot{\ph}\in H^1(C,\vc{p};\CC)\cong T_{(C,\vc{p},\ph)}\Omega\M'_{g,n}(\vc{m})$, then
\[
A(\ph+\e \dot{\ph})=A(\ph)+\e b(\dot{\ph},\ph)+\ol{\e}b(\ph,\dot{\ph})+
|\e|^2 b(\dot{\ph},\dot{\ph})
\]
Claims (c) and (d) for $A$ then follow from Lemma \ref{lemma:b}.

As for $\log(A)$, we compute
$
i\pa\ol{\pa}\log(A)=i A^{-2}(\pa\ol{\pa}A-\pa A\wedge\ol{\pa}A)
$
and so
\[
i\pa\ol{\pa}\log(A)_\ph(\dot{\ph},\dot{\ph})=
\frac{A(\ph)A(\dot{\ph})-|h(\ph,\dot{\ph})|^2}{A(\ph)^2}
\]
from which we conclude that,
as a Hermitian form on $H^1(C;\CC)$,
the Hessian $i\pa\ol{\pa}\log(A)$
is negative-definite on $H^{0,1}(C)$, vanishes along the line $\CC\cdot(\ph)$
and it is positive-definite on $H^{1,0}(C)\cap(\ph)^\perp$, which shows that
$\log(A)$ satisfies condition weak $(\star)$.
As $i\pa\ol{\pa}\log(A)$ has signature $(g-1,g,1)$
as a Hermitian form on $H^1(C;\CC)$, it follows that 
it has signature $(g-1,g,k)$
as a Hermitian form on $H^1(C,\vc{p};\CC)$.
Hence, (c) and (d) are proven.
\end{proof}

\subsection{Length functions}\label{sec:length}

Since degenerations on flat surfaces of bounded area occur as conical singularities merge together, it makes sense to study distance functions between couples of marked points along prescribed paths.

\begin{definition}
An {\it{arc}} on the surface $C$ with marked points $\vc{p}$ is a nontrivial homotopy class (with fixed endpoints) of a simple oriented path $\gamma$ on $C$
that intersects $\vc{p}$ at its endpoints. A set of arcs $B=\{\gamma_j\}$ is a {\it{basis}} if
the set of homology classes $\{(\gamma_j)\}$ forms a basis of the real vector space $H_1(C,\vc{p};\RR)$.
We denote by $\Bcal$ the set of all bases of arcs.
\end{definition}

Let $\ph$ be a non-zero Abelian differential on $(C,\vc{p})$
and assume that the set of its zeros $Z(\ph)$ is contained inside $\vc{p}$.
The associated metric $|\ph|^2$ on $C$ is flat
outside $Z(\ph)$ and 
it has a conical singularity of angle $2\pi(d+1)$ at a zero
of order $d$.

\begin{definition}
A {\it{segment}} for $(C,\vc{p},\ph)\in\Omega\M'_{g,n}(\vc{m})$
is a smooth $\ph$-geodesic path between two
(not necessarily distinct) points in $\vc{p}$.
We will say that a segment is {\it{proper}} if it intersects
$\vc{p}$ at its endpoints only.
\end{definition}

\begin{remark}
Such segments joining two zeros of an Abelian or, more generally,
a holomorphic quadratic differential are often called
{\it{saddle connections}}, the terminology being borrowed from the world of minimal surfaces.
\end{remark}

As the metric $|\ph|^2$ is non-positively curved,
there always exists a unique geodesic representing a given arc $\gamma$. Such a geodesic is a concatenation of proper
segments $\gamma=\gamma^{(1)}\cup\dots\cup\gamma^{(s)}$.
We denote by $\ell_\gamma(\ph)$ the {\it{length}} of such a geodesic with respect to the metric $|\ph|^2$. Clearly,
$\ell_\gamma(\ph)=\ell_{\gamma^{(1)}}(\ph)+\dots+\ell_{\gamma^{(s)}}(\ph)$.

\begin{definition}
The {\it{systole}} $\ell_{sys}(\ph)$ of $(C,\vc{p},\ph)\in\Omega\M'_{g,n}(\vc{m})$ is the 
infimum of the lengths $\ell_\gamma(\ph)$, as $\gamma$ ranges over all non-trivial
arcs in $C$.
\end{definition}

We recall some preliminary results
on the geometry of such flat surfaces with conical singularities.

The content of Lemma \ref{lemma:ell-properties0} is rather well-known (see for instance \cite{kms:ergodicity} and \cite{masur-smillie:hausdorff}, or \cite{masur:ergodic}) and implicitly or explicitly contained in all papers that deal with counting problems of saddle connections on translation surfaces. Similar considerations hold for Lemma \ref{lemma:ell-properties}, whose conclusions can be easily deduced from the existing literature on the subject. We claim no originality for these two statements, but proofs are included for completeness.

\begin{lemma}[Systole and short arcs]\label{lemma:ell-properties0}
For every $(C,\vc{p},\ph)$ in $\Omega\M'_{g,n}(\vc{m})$, the following hold:
\begin{itemize}
\item[(a)]
for every $\ell>0$ there are finitely many $\ph$-segments
on $(C,\vc{p})$ of length less than $\ell$;
\item[(b)]
the systole $\ell_{sys}(\ph)$
is always realized at a proper segment.
\end{itemize}
\end{lemma}
\begin{proof}
Let $\tilde{C}\rar C$ be the universal cover and let
$\tilde{\ph}$ be the pullback of $\ph$ to $\tilde{C}$.
Denote by $\tilde{\vc{p}}\subset\tilde{C}$ the subset of points
that project to $\vc{p}\subset C$.

Let $D\subset\tilde{C}$ be a fundamental domain and $R$ its
diameter (with respect to the $|\tilde{\ph}|^2$ metric). Fix a point
$\tilde{x}\in D$. Then each segment of length at most $\ell$
is contained in the ball $B(\tilde{x},\ell+R)\subset \tilde{C}$.
Since $C$ is compact and so its systole
is positive, $B(\tilde{x},\ell+R)$ intersects only finitely many
translates of $D$ and so $B(\tilde{x},\ell+R)\cap \tilde{\vc{p}}$ is finite.
Thus, there are only finitely many segments completely contained in $B(\tilde{x},\ell+R)$ and this proves (a).

As a consequence, the systole of $(C,\vc{p},\ph)$ is certainly
attained at some arc $\gamma$.
Write the $\ph$-geodesic representative of $\gamma$ as a
concatenation of proper segments $\gamma^{(1)}\cup\dots\cup\gamma^{(s)}$.
If $s>1$, then $\ell_{\gamma^{(1)}}(\ph)<\ell_\gamma(\ph)=
\ell_{sys}(\ph)$ would give a contradiction. Hence, $s=1$ and
$\gamma$ is a proper segment for $(C,\vc{p},\ph)$, which proves (b).
\end{proof}

Let $(C,\vc{p},\ph)\in\Omega\M'_{g,n}(\vc{m})$ 
and consider $U\subset \Omega\M'_{g,n}(\vc{m})$ a small contractible open neighbourhood of $(C,\vc{p},\ph)$,
so that the universal family
of marked Riemann surfaces can be $C^\infty$ identified to
$(C,\vc{p})\times U$
over $U$ and 
$\Pcal$ maps $U$ biholomorphically onto an open subset of $H^1(C,\vc{p};\CC)$.
We can think of points of $u\in U$ as of closed
differential $1$-forms $\ph_u$ on the fixed marked surface $(C,\vc{p})$.

We recall that,
for every $(\gamma)\in H_1(C,\vc{p};\ZZ)$, the {\it{local period}}
\[
\Pcal_\gamma: U \rar \CC
\]
defined as $\Pcal_\gamma(u):=\int_\gamma \ph_u$ is holomorphic by Proposition \ref{prop:period-closed-hol}.
Moreover, if $\gamma=\gamma^{(1)}\cup\dots\cup\gamma^{(s)}$ is 
a decomposition of $\gamma$ in a union of $\ph_u$-segments,
then $|\Pcal_{\gamma^{(i)}}(u)|=\ell_{\gamma^{(i)}}(u)$
and $\ell_\gamma(u)=\sum_i \ell_{\gamma^{(i)}}(u)=
\sum_i |\Pcal_{\gamma^{(i)}}(u)|\geq |\Pcal_\gamma(u)|$.

\begin{notation}
Let $U_\gamma\subset U$ the locus of differentials for which $\gamma$
can be represented by a proper segment and $U_B:=\bigcap_{\gamma\in B}U_\gamma$ the locus of differentials for which
$B$ is a basis made of proper segments.
\end{notation}

\begin{lemma}[Short bases are made of proper segments]\label{lemma:ell-properties}
The following properties hold.
\begin{itemize}
\item[(a)]
For every arc $\gamma$ and basis $B$,
the loci $U_\gamma\subset U$ and $U_B\subset U$ are open.
\item[(b)]
Let $f: \RR_+\rar\RR_{\geq 0}$ be a strictly decreasing continuous function.
Then
\begin{itemize}
\item[(b1)]
the function
$\dis
f_\Bcal:=\sup_{B\in \Bcal} f_B : U\lra\RR_{\geq 0}
$
is well-defined, where
$ 
f_B:=\sum_{\gamma\in B} (f\circ\ell_\gamma)$;
\item[(b2)]
the function $f_\Bcal$
is locally the maximum of finitely many $f_B|_{U_B}$;
\item[(b3)]
if the value $f_\Bcal(u)$ is attained at a basis $B\in\Bcal$,
then $B$ contains an arc that realizes the systole for $|\ph_u|^2$
and so $f_\Bcal(u)\geq f(\ell_{sys}(\ph_u))$.
\end{itemize}
\end{itemize}
\end{lemma}
\begin{proof}
To prove (a), it is clearly enough to prove that $U_\gamma$ is a neighbourhood of each $u\in U_\gamma$.

Let $\tilde{C}\rar C$ be the universal cover and denote by
$\tilde{\vc{p}}\subset \tilde{C}$ the points that map to $\vc{p}$.
The arc $\gamma$ lifts to an arc $\tilde{\gamma}$
with distinct endpoints and every $\ph_u$ on $C$ pulls back
to a $\tilde{\ph}_u$ on $\tilde{C}$.
Call $\tilde{\gamma}_u$ the $\tilde{\ph}_u$-geodesic on $\tilde{C}$ homotopic to $\tilde{\gamma}$.

Since $\tilde{\gamma}_u$ is a proper segment, there exists
an $\e>0$ such that the $\e$-neighbourhood $N_\e$ of $\tilde{\gamma}_u$
with respect to $|\tilde{\ph}_u|^2$ does not contain any other
point of $\tilde{\vc{p}}$.

Thus, there exists a neighbourhood $U'\subset U$ of $u$
such that $\ell_{\gamma}(\ph_{u'})<\ell_\gamma(\ph_u)+\e/4$
and $dist_{|\ph_{u'}|^2}(\tilde{\gamma}_u,\pa N_\e)>\e/2$
for all $u'\in U'$. Hence, $\gamma_{u'}$ is realized by a
$|\ph_{u'}|^2$-geodesic contained inside $N_\e$, which is thus
a proper segment, for all $u'\in U'$. This shows that $U'\subset U_\gamma$ and so $U_\gamma$ is open.

In order to prove (b1) and (b2), we want to show that
\begin{itemize}
\item[(i)]
for every $u\in U$ there exist finitely many bases $B_1,\dots,B_r\in\Bcal$ such that
$f_\Bcal(u)=f_{B_j}(u)$;
\item[(ii)]
each arc in each $B_j$ is realized by a proper $\ph_u$-segment (and so $u\in U_{B_1}\cap\dots\cap U_{B_r}$);
\item[(iii)]
there exists a relatively compact neighbourhood $U''\subset U_{B_1}\cap\dots\cap U_{B_r}$ of $u$ such that $f_\Bcal|_{_{U''}}=\max_j f_{B_j}|_{_{U''}}$.
\end{itemize}
%

Pick a compact neighbourhood $K\subset U$ of $u$.
By Lemma \ref{lemma:ell-properties0}(a) and the
fact that $f$ is strictly decreasing, there exists finitely many
bases $B_1,\dots,B_r,B_{r+1},\dots,B_{r+s}\in\Bcal$ such that
$f_{\Bcal}|_{_{K}}=\max_{j=1}^{r+s} f_{B_j}|_{_{K}}$, which already
proves Claim (i).
Without loss of generality, we can assume that $f_\Bcal(u)=f_{B_j}(u)$
for $j=1,\dots,r$ and that $f_{\Bcal}(u)>f_{B_j}(u)$ for $j=r+1,\dots,r+s$.
Thus, up to choosing a smaller neighbourhood $K$ of $u$, we
can assume $s=0$ and so Claim (iii) will follow from Claim (ii) by choosing $U''=K\cap U_{B_1}\cap\dots\cap U_{B_r}$.

As for claim (ii), we proceed by contradiction:
we assume that $B_j=\{\gamma_1,\dots,\gamma_{2g+n+k-1}\}$ and that $\gamma_i$ is realized by a $\ph_u$-geodesic which is not a proper segment.
Then $\gamma_i$ is a nontrivial concatenation
of two paths $\gamma'_i$ and $\gamma''_i$, where $\gamma'_i$ and
$\gamma''_i$ are union of proper segments, so that
$\ell_{\gamma'_i},\ell_{\gamma''_i}<\ell_{\gamma_i}$.

Since $B_j$ is a basis of $H_1(C,\vc{p};\RR)$,
we can write $(\gamma'_i)=\sum_h a'_h (\gamma_h)$ and
$(\gamma''_i)=\sum_h a''_h(\gamma_h)$.
As $a'_i+a''_i=1$, at least one of the two must be non-zero: without loss of generality, we can assume that $a'_i\neq 0$.
Then $B'_j:=(B_j\setminus\{\gamma_i\})\cup\{\gamma'_i\}$ is again a basis of
arcs, but $f_{B'_j}(u)>f_{B_j}(u)=f_\Bcal(u)$, against the definition of $f_{\Bcal}$.

%
%
%

A similar argument is used to prove (b3). Indeed, assume by
contradiction that $\ell_{\gamma_i}(\ph_u)>\ell_{sys}(\ph_u)$
for all $\gamma_i\in B_j$ 
and pick an arc $\gamma'$ that realizes the systole for $|\ph_u|^2$.
By Lemma \ref{lemma:ell-properties0}(b), such $\gamma'$ is a
$\ph_u$-segment and so 
$0<\ell_{\gamma'}(\ph_u)=\int_{\gamma'}\ph_u$, which implies that
$(\gamma')\neq 0$ in $H_1(C,\vc{p};\RR)$.
Thus, $\{\gamma'\}$ can be completed to an $\RR$-basis $B'_j$ of $H_1(C,\vc{p};\RR)$ using
elements of $B_j$: in particular $B'_j=\{\gamma'\}\cup\{\gamma_i\,|\,i\neq i'\}$ for some $i'$.
Since $\ell_{\gamma'}(\ph_u)<\ell_{\gamma_i}(\ph_u)$ for all $i$ and $f$ is strictly
decreasing, it follows that $f_{B'_j}(\ph_u)>f_{B_j}(\ph_u)=f_{\Bcal}(\ph_u)$ against the definition of $f_\Bcal$.
\end{proof}

As the $f_\Bcal$ above defined in
Lemma \ref{lemma:ell-properties}(b) is the sup of {\it{all}} bases, it is clearly independent of the local
trivialization of the family and so it globalizes over
$\Omega\M'_{g,n}(\vc{m})$. Moreover, it is invariant under labelling of the marked points and so
we have obtained the following.

\begin{corollary}[Functions of arc lengths]\label{cor:functions-lengths}
For every continuous strictly decreasing $f:\RR_+\rar\RR_{\geq 0}$, there is a global
$f_\Bcal:\Omega\M_{g,n}(\vc{m})\rar\RR_{\geq 0}$, which is locally the maximum of finitely many
functions which are as regular as $f$,
and such that $f_\Bcal\geq f\circ\ell_{sys}$.
Moreover, if $f$ is homogeneous of degree $d$, then $f_\Bcal$ is
homogeneous of degree $d$.
\end{corollary}

The decreasing function we will be interested in is
the {\it{inverse square length functional}} $\ell^{-2}_\Bcal: \Omega\M_{g,n}(\vc{m})\rar\RR_+$ associated to $f(x):=x^{-2}$.

\begin{lemma}[Inverse square length functional]\label{lemma:ell-properties2}
The function $\ell^{-2}_\Bcal$
has the following properties:
\begin{itemize}
\item[(a)]
$\ell^{-2}_\Bcal$ is homogeneous of degree $-2$;
\item[(b)]
$\ell^{-2}_\Bcal$ is strictly plurisubharmonic on $\Omega\M_{g,n}(\vc{m})$, thus
$\log(\ell^{-2}_\Bcal)$ is plurisubharmonic on $\Omega\M_{g,n}(\vc{m})$ and
it is strictly plurisubharmonic when restricted to a codimension $1$ submanifold
transverse to the rays of $\Omega\M_{g,n}(\vc{m})$;
\item[(c)]
$\ell^{-2}_\Bcal$ and $\log(\ell^{-2}_\Bcal)$ satisfy condition $(\star)$.
\end{itemize}
\end{lemma}
\begin{proof}
Part (a) is obvious. Claim (c) immediately follows from (b).

In order to prove (b), we can work on 
the moduli space $\Omega\M'_{g,n}(\vc{m})$ of Abelian differentials with marked zeros, because the map 
$\Omega\M'_{g,n}(\vc{m})\rar\Omega\M_{g,n}(\vc{m})$ 
that forgets the last $k$ markings
is a finite \'etale cover.
By Lemma \ref{lemma:ell-properties}(b), 
every $(C,\vc{p},\ph)\in\Omega\M'_{g,n}(\vc{m})$
has a neighbourhood on which
$\ell^{-2}_\Bcal$ can be written as the maximum
of $\ell_{B_1}^{-2},\dots,\ell_{B_r}^{-2}$ and
the arcs in $B_1,\dots,B_r$ are realized by proper segments
with respect to every Abelian differential in such a neighbourhood.
Hence, $\ell_{B_j}^{-2}=
\sum_{\gamma\in B_j} |\Pcal_\gamma|^{-2}$,
which is strongly plurisubharmonic, because $B_j$ is a basis.
Being locally the maximum of finitely many strongly plurisubharmonic functions, $\ell^{-2}_\Bcal$ is strongly plurisubharmonic too.
It is straightforward to check that
$i\pa\ol{\pa}\log(\ell_{B_j}^{-2})$ is semipositive-definite and its radical
coincides with the direction given by rescaling. 
\end{proof}

As the area functional is $2$-homogeneous and $\ell_{sys}$ is $1$-homogeneous, the ratio $\frac{A(\ph)}{\ell_{sys}^2(\ph)}$ 
is invariant under rescaling of the Abelian differential $\ph$. 
The following is a rephrasing of Proposition 1 in \cite{kms:ergodicity}. 

\begin{proposition}[Behaviour of diverging sequences]\label{prop:diverging}
Let $(C_s,[\ph_s])_{s\in\NN}$ be a divergent sequence
in the moduli space $\PP\Omega\M_{g,n}(\vc{m})$
of projective Abelian differentials.
Then either $\liminf\ell_{sys}(\ph_s)\rar 0$ or $\limsup A(\ph_s)\rar\infty$
as $s\rar\infty$.
Hence, the function $\frac{A}{\ell^2_{sys}}:\PP\Omega\M_{g,n}(\vc{m})\rar\RR_{\geq 0}$ is well-defined and proper.
The same conclusion holds for the induced function on the moduli space $\PP\Omega\M'_{g,n}(\vc{m})$ of Abelian differentials with marked zeros.
\end{proposition}


\subsection{Cohomological dimension of strata}


Since $A$ and $\ell_\Bcal^{-2}$ are homogeneous of degrees $2$ and $-2$ respectively,
the function
\[
\wti{\exh}_{\vc{m}}:=\log(A\cdot\ell^{-2}_{\Bcal}):\Omega\M'_{g,n}(\vc{m})\lra\RR
\]
is $\CC^*$-invariant and so it descends to a well-defined
\[
\exh_{\vc{m}}:\PP\Omega\M'_{g,n}(\vc{m})\lra\RR
\]

%
%

We are now ready to prove the following.

\begin{proposition}[Strata are $(g+1)$-complete]\label{prop:function}
The above defined $\exh_{\vc{m}}$ is an $\aut_n(\vc{m})$-invariant
exhaustion function on the moduli space
$\PP\Omega\M'_{g,n}(\vc{m})$ of Abelian differentials
with marked zeros
and it
satisfies condition ($\star$).
Hence, $\exh_{\vc{m}}$ is
strongly $(g+1)$-convex.
\end{proposition}

\begin{proof}
It is immediate to see that $\exh_{\vc{m}}\geq \frac{A}{\ell^2_{sys}}$,
because $\ell^{-2}_{\Bcal}\geq \ell^{-2}_{sys}$
by Corollary \ref{cor:functions-lengths}.
Since $\frac{A}{\ell^2_{sys}}$ is an exhaustion function on $\PP\Omega\M'_{g,n}(\vc{m})$ due to
Proposition \ref{prop:diverging}, the same holds for $\exh_{\vc{m}}$.
Moreover, both $A$ and $\ell^{-2}_\Bcal$ are $\aut_n(\vc{m})$-invariant, and so is $\exh_{\vc{m}}$.
%
Now consider
\[
i\pa\ol{\pa}\wti{\exh}_{\vc{m}}=
i\pa\ol{\pa}\log(A)+i\pa\ol{\pa}\log(\ell^{-2}_\Bcal)
\]
We recall that, by Lemma \ref{lemma:A}(d)
and Lemma \ref{lemma:ell-properties2}(c),
the functions $\log(A)$ and $\log(\ell^{-2}_\Bcal)$
satisfy condition weak $(\star)$ and condition $(\star)$ respectively.
Hence, $\wti{\exh}_{\vc{m}}$ satisfies condition $(\star)$ by
Lemma \ref{lemma:star-properties}(d) and so
$\exh_{\vc{m}}:\PP\Omega\M'_{g,n}(\vc{m})\rar\RR$ is strongly $(g+1)$-convex.
\end{proof}

As a consequence, we obtain our first main result.

\begin{proof}[Proof of Theorem \ref{mthm:strata}]
We recall that, given $\sigma\in\Sfrak$,
the stratum $\PP\Omega\M'_{g,n}(\sigma)$ can be
identified to $\PP\Omega\M'_{g,n}(\vc{m})$
with $\vc{m}=\sigma_*\vc{\mo}$.
Moreover, the moduli space $\PP\Omega\M'_{g,n}(\vc{m})$
of projective Abelian differentials 
with marked zeros of orders $\vc{m}$
has $\cohdim_{Dol}\leq g$
by Proposition \ref{prop:function} and
Theorem \ref{thm:q-convex}.
\end{proof}

Since the map $\PP\Omega\M'_{g,n}(\vc{m})\rar\PP\Omega\M_{g,n}(\vc{m})$ that forgets the last $k$ markings 
is a finite \'etale and surjective cover,
we conclude that $\cohdim_{Dol}(\PP\Omega\M_{g,n}(\vc{m}))\leq g$ too
by Lemma \ref{lemma:fibr}.

\section{Thickening the strata}\label{sec:thickening}

We have seen the moduli space $\PP\Omega\M'_{g,n}$
of projective Abelian differentials with marked zeros
is stratified by the locally closed submanifolds $\PP\Omega\M'_{g,n}(\sigma)$
of codimension $d(\sigma)$, where the surjection $\sigma\in\Sfrak$ describes
how the marked zeros of the differential coalesce.

It is intuitively clear that 
$(C,\vc{p},[\ph])\in\PP\Omega\M'_{g,n}$ lying in
a small neighbourhood $V_\sigma$ of $\PP\Omega\M'_{g,n}(\sigma)$
is represented by a flat surface on which,
for every $j$, the points 
$\vc{p}(\sigma^{-1}(j))$ are close to one another, namely they belong to a small disk
$\disk^j_\sigma(\ph)\subset C$.
As $(C,\vc{p},[\ph])$ approaches the stratum $\PP\Omega\M'_{g,n}(\sigma)$,
the points $\vc{p}(\sigma^{-1}(j))$ tend to coalesce and the disks $\disk^j_\sigma(\ph)$ can be chosen to be smaller and smaller. Moreover, it seems reasonable that the disks $\disk^j_\sigma(\ph)$ can be kept disjoint, by picking $V_\sigma$ thin enough.

In the first part of the current section we will make this idea precise and we will produce
a cover $\Vfrak=\{V_\sigma\,|\,\sigma\in\Sfrak\}$
adapted to the stratification $\{\PP\Omega\M'_{g,n}(\sigma)\,|\,\sigma\in\Sfrak\}$
in the sense of Definition
\ref{def:adapted}.

Then we will study certain geometric functions on the thickenings $V_\sigma$ in a similar fashion as we did for locally closed strata,
since their complex Hessian is again easily computable in period coordinates.

The first function we wish to consider, $\eta_\sigma$, is essentially an extension of the exhaustion function
on $\PP\Omega\M'_{g,n}(\sigma)$ we defined in the previous section, 
and so in particular $\eta_\sigma$ is proper in directions ``parallel to the stratum''. Such an $\eta_\sigma$ is
obtained by combining length functions of short arcs that join the disks $\disk^j_\sigma$.
The second function $\zeta_\sigma$ is constructed using
periods of long paths contained inside the disks $\disk^j_\sigma$. In order to obtain an exhaustion
function, it is necessary to consider a combination $\xi_\sigma=\eta_\sigma+\chi\circ\zeta_\sigma$, where
$\chi$ is a suitable convex function that diverges very quickly and to restrict the domain to a smaller
subset $W_\sigma\subset V_\sigma$.

The desired estimate for $\cohdim_{Dol}(\PP\Omega\M'_{g,n})$
will follow by computing the dimension of the nerve of the cover $\Wfrak=\{W_\sigma\}$
and by showing that each $\xi_\sigma$ satisfies property $(\star)$ and so
$\xi_{\sigma_0}+\dots+\xi_{\sigma_d}$ endows
the intersection $W_{\sigma_0}\cap\cdots\cap W_{\sigma_d}$ (whenever non-empty)
with a strongly $(g+1)$-convex exhaustion function.

\subsection{Merging data and good covers}

We first formalize what it means for the disks $\disk^j_\sigma(\ph)\subset C$
to move in families as $(C,\vc{p},\ph)$ varies in $V_\sigma$
and what further properties we require from such disks.


\begin{definition}\label{def:clashing}
Let $\Vfrak=\{V_\sigma\,|\,\sigma\in\Sfrak\}$ be an open cover 
of $\PP\Omega\M'_{g,n}$ adapted to the stratification in
the sense of Definition \ref{def:adapted}.
Denote by $\C_{V_\sigma}\rar V_\sigma$
the induced tautological family of Riemann surfaces.
A {\it{merging datum}} for $V_\sigma$ is
a collection $\{\disk^j_\sigma
\,|\, 1\leq j\leq n+2g-2-d(\sigma)\}$ 
of open subsets of $\C_{V_\sigma}$
that satisfies the following properties:
\begin{itemize}
\item[(i)]
$\disk^j_\sigma\rar V_\sigma$ is a disk bundle for all
$j=1,\dots,n+2g-2-d(\sigma)$;
\item[(ii)]
$\disk^j_\sigma\cap \disk^{j'}_\sigma=\emptyset$
if $j\neq j'$;
\item[(iii)]
the marking $\vc{P}$ restricts to
$V_\sigma\times\sigma^{-1}(j)\rar \disk^j_\sigma
\subset\C_{V_\sigma}$
for every $j=1,\dots,n+2g-2-d(\sigma)$;
\item[(iv)]
at each $(C,\vc{p},[\ph])\in V_\sigma$, proper $\ph$-segments between
marked points
contained in
\[
\disk_\sigma(\ph)=\bigcup_{j=1}^{n+2g-2-d(\sigma)} 
\disk^j_\sigma(\ph),
\quad
\text{where}
\ \disk^j_\sigma(\ph):=
\disk^j_\sigma\Big|_{(C,\vc{p},[\ph])},
\]
generate
$H_1(\disk_\sigma(\ph),\vc{p};\RR)$. Call {\it{inner}} the arcs homotopic to such segments
and {\it{outer}} all the others.
\end{itemize}
Given $c>0$, the collection $\{\disk^j_\sigma\}$ is called a {\it{$c$-merging datum}}
if $\frac{\ell_\beta(\ph)}{\ell_{\gamma}(\ph)}<c$ for every
$(C,\vc{p},[\ph])\in V_\sigma$ and for every
inner arc $\beta$ and outer arc $\gamma$ on $(C,\vc{p},[\ph])$.
\end{definition}

\begin{definition}
A {\it{good cover}} of $\PP\Omega\M'_{g,n}$ is an adapted cover $\Vfrak$
that admits a $c$-merging datum for
each $V_\sigma\in\Vfrak$ and some fixed $c\leq c_0(g,n):=\frac{1}{2g+n}$.
\end{definition}

\begin{center}
\begin{figurehere}
\psfrag{C}{$C$}
\psfrag{N1}{$\disk^1_\sigma$}
\psfrag{N2}{$\disk^2_\sigma$}
\psfrag{N3}{$\disk^3_\sigma$}
\psfrag{p1}{$p_1$}
\psfrag{p2}{$p_2$}
\psfrag{p3}{$p_3$}
\psfrag{p4}{$p_4$}
\psfrag{p5}{$p_5$}
\psfrag{b}{$\beta$}
\psfrag{g}{$\gamma$}
\includegraphics[width=0.6\textwidth]{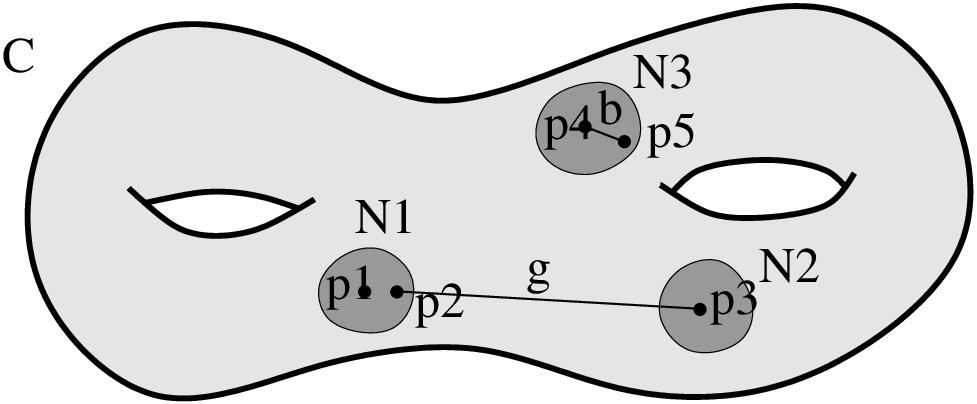}
\caption{{\small Example of merging datum corresponding to $\sigma(1)=\sigma(2)=1$, $\sigma(3)=2$ and $\sigma(4)=\sigma(5)=3$ with an inner $\beta$ and an outer $\gamma$ segment.}}\label{fig:clash}
\end{figurehere}
\end{center}

\begin{remark}
Property (iv) in Definition \ref{def:clashing} is essentially a convexity property for the disks with respect to the flat metric. The existence of a $c$-merging datum with a small costant $c$, as required for a good cover, makes the surfaces $(C,\vc{p},[\ph])\in V_\sigma$ look like as in Figure \ref{fig:clash}
and forces $V_\sigma$ to be thin enough. The exact value of $c_0$
can be understood in view of Remark \ref{rmk:short-geodesics}.
\end{remark}

We begin with a preliminary observation.

\begin{lemma}\label{lemma:refinement}
Every refinement 
$\Wfrak=\{W_\sigma\}$
of a good cover $\Vfrak=\{V_\sigma\}$
such that $W_\sigma$ contains 
the stratum $\PP\Omega\M'_{g,n}(\sigma)$
%
is also good.
\end{lemma}
\begin{proof}
Since the restriction of a $c$-merging datum for $V_\sigma$ to a smaller open subset of
$\PP\Omega\M'_{g,n}$ is still a $c$-merging datum, it is enough
to show that $\Wfrak$ is adapted to the stratification.

Now, property (AS1) in Definition \ref{def:adapted} is satisfied by assumption; thus, we only have to check property (AS2).

If $\sigma\preceq\tau$, then $W_\sigma\cap \PP\Omega\M'_{g,n}(\tau)\neq\emptyset$ and so $W_\sigma\cap W_\tau\neq\emptyset$
by (AS1).
On the other hand, if $\sigma\npreceq\tau$ and
$\tau\npreceq\sigma$, then $W_\sigma\cap W_\tau$ is contained inside
$V_\sigma\cap V_\tau=\emptyset$ and so it is empty.
This proves (AS2) for $\Wfrak$ and so the second claim.
\end{proof}

The aim of this subsection is to show the following.

\begin{proposition}[Existence of good covers]\label{prop:good}
Good covers of 
the moduli space
$\PP\Omega\M'_{g,n}$ of projective Abelian differentials
with marked zeros exist.
\end{proposition}

Clearly, it is enough to work with
$\CC^*$-invariant subsets of $\Omega\M'_{g,n}$ and to produce
a cover $\wti{\Vfrak}=\left\{\wti{V}_\sigma\right\}$ 
of the unprojectivized moduli space $\Omega\M'_{g,n}$
that descends to the wished good cover $\Vfrak=\{V_\sigma\}$ of $\PP\Omega\M'_{g,n}$.

\begin{definition}
The {\it{injectivity radius}} {\it{relative to ${\sigma}$}}
is the function
$R_\sigma:\Omega\M'_{g,n}\rar\RR_{\geq 0}$ defined as
\[
R_\sigma(C,\vc{p},\ph):=\frac{1}{2}\mathrm{min}
\Big\{
dist_{\ph}(p_i,p_{i'})
\ \Big|
\ \sigma(i)\neq\sigma(i')
\Big\}
\cup
\left\{
\ell_{\gamma}(\ph)\,\Big|\,
\text{$\gamma$ nontrivial loop}
\right\}
\]
\end{definition}

Notice that $R_\sigma$ is $1$-homogeneous, it
vanishes exactly
on those $\Omega\M'_{g,n}(\tau)$ such that $\sigma\npreceq\tau$,
and it is strictly positive elsewhere.\\

{\bf{Definition of the disks.}}
Let $c\in (0,1)$ be a parameter to be determined later.

Fix $\sigma\in\Sfrak$.
For every $j=1,\dots,n+2g-2-d(\sigma)$ denote by
$i_j\in[1,n+2g-2]$ the smallest
integer such that $\sigma(i_j)=j$.

For every $(C,\vc{p},\ph)\in \Omega\M'_{g,n}$ at which $R_\sigma(\ph)>0$,
denote by $\widetilde{\disk}^j_\sigma(\ph)\subset C$ the open $\ph$-ball centered at $p_{i_j}$
of radius 
$r_\sigma(\ph):=\frac{c}{2}R_\sigma(\ph)$,
which is a topological $2$-disk and let $\widetilde{\disk}^j_\sigma\subset\Omega\C'_{g,n}$ be the union of all such disks.

\begin{remark}\label{rmk:short-geodesics}
Since the metric $|\ph|^2$ is non-positively curved, the disk
$\wti{\disk}^j_\sigma(\ph)$ is $\ph$-convex in the following sense: for every pair of
points $q,q'\in \wti{\disk}^j_\sigma(\ph)$,
there always exists a unique $\ph$-geodesic $\beta$ of length $<2r_\sigma(\ph)$
completely contained inside $\wti{\disk}^j_\sigma(\ph)$
that joins $q$ and $q'$. Such a $\beta$ is a concatenation of segments
$\beta=\beta^{(1)}\cup\cdots\cup\beta^{(s)}$ and it
may pass through a marked point at most once,
because $\wti{\disk}^j_\sigma(\ph)$ is simply-connected.
Since there are at most $n+2g-2$ distinct marked points, necessarily $s\leq n+2g-3$.
If $q,q'$ are marked points,
then each $\beta^{(e)}$ is an inner segment.
Hence, the condition $\ell_{\beta^{(e)}}(\ph)<c\,\ell_{\gamma}(\ph)$ for all
the inner segments $\beta^{(e)}$ and the outer arcs $\gamma$ with $c\leq c_0$
implies the all (possibly non-smooth) geodesics contained in
$\wti{\disk}^j_\sigma(\ph)$ and joining marked points are shorter than any outer segment, because $c_0 s<1$.
\end{remark}

Now consider the open $\CC^*$-invariant neighbourhood 
\[
\wti{V}_\sigma:=\big\{(C,\vc{p},\ph)\in\Omega\M'_{g,n}\ \big|\ p_i\in\widetilde{\disk}^j_\sigma(\ph)
\ \text{whenever $\sigma(i)=j$}\big\}
\]
of the locally closed stratum $\Omega\M'_{g,n}(\sigma)$ inside $\Omega\M'_{g,n}$
and we define
\[
\wti{\disk}^j_\sigma:=\bigcup_{(C,\vc{p},\ph)\in\wti{V}_\sigma}\wti{\disk}^j_\sigma(\ph)
\]
so that the induced ${\disk}^j_\sigma\rar {V}_\sigma$ is evidently a disk bundle that satisfies property (iii) of Definition \ref{def:clashing}.

\begin{lemma}\label{lemma:disjoint}
The disk bundles ${\disk}^j_\sigma$ are disjoint.
\end{lemma}
\begin{proof}
The statement is equivalent to the disjointness
of the disk bundles $\wti{\disk}^j_\sigma$.

By contradiction, suppose that there exists a point $q\in
\wti{\disk}^j_\sigma\cap \wti{\disk}^{j'}_\sigma$ for some $j\neq j'$,
corresponding to some $(C,\vc{p},\ph)\in \wti{V}_\sigma$, and
let $i=i_j$ and $i'=i_{j'}$.
By definition of $R_\sigma$, we have
\[
R_\sigma(\ph)\leq \frac{1}{2}dist_{\ph}(p_i,p_{i'})\leq \frac{1}{2}dist_{\ph}(p_i,q)+\frac{1}{2}dist_{\ph}(q,p_{i'})
\]
On the other hand, by the choice of the radii of the disks
\[
\frac{1}{2}dist_{\ph}(p_i,q)+\frac{1}{2}dist_{\ph}(q,p_{i'})<
r_\sigma(\ph)<R_\sigma(\ph)
\]
and the claim is proven.
\end{proof}

\begin{lemma}\label{lemma:adapted}
The cover ${\Vfrak}=\{{V}_\sigma\}$ is adapted to the stratification of $\PP\Omega\M'_{g,n}$.
\end{lemma}
\begin{proof}
It is enough to show that $\wti{\Vfrak}=\{\wti{V}_\sigma\}$
is adapted to the stratification of $\Omega\M'_{g,n}$.

%
Clearly, $\Omega\M'_{g,n}(\sigma)\subset \wti{V}_\sigma$
for all $\sigma\in\Sfrak$ and so the property (AS1) is satisfied. Moreover,
$\wti{V}_\sigma\cap \wti{V}_\tau\neq\emptyset$ if
$\sigma\preceq\tau$ or $\tau\preceq\sigma$.

In order to show that (AS2) holds, let's proceed by contradiction and assume
that there exists $(C,\vc{p},\ph)\in \wti{V}_\sigma\cap \wti{V}_\tau$ with $\sigma\npreceq\tau$ and $\tau\npreceq\sigma$.
Since $\sigma\npreceq \tau$, there exist $i,i'\in\{1,\dots,n+2g-2\}$ such that
$\sigma(i)\neq \sigma(i')$ but $\tau(i)=\tau(i')$.
The first condition implies that $R_\sigma(\ph)\leq \frac{1}{2}dist_{\ph}(p_i,p_{i'})$ and the second condition implies that 
$dist_{\ph}(p_i,p_{i'})<2r_\tau(\ph)$,
so that $\frac{2}{c} R_{\sigma}(\ph)< R_\tau(\ph)$.
Vice versa, $\tau\npreceq\sigma$ implies $\frac{2}{c} R_\tau(\ph)<R_\sigma(\ph)$.
Hence, $\left(\frac{2}{c}\right)^2 R_\sigma(\ph)<R_\sigma(\ph)$, which gives the desired contradiction.
\end{proof}

\begin{lemma}\label{lemma:clashing}
For all $\sigma\in\Sfrak$,
the disk bundles $\{{\disk}^j_\sigma\rar {V}_\sigma\,|\,
j=1,\dots,n+2g-2-d(\sigma)\}$ endow
${V}_\sigma$ with a $c$-merging datum.
\end{lemma}
\begin{proof}
After the above discussion, we are only left to verify
property (iv) and that
$\{{\disk}^j_\sigma\}$ is a $c$-merging datum.
As usual, we work with the disk bundles $\wti{\disk}^j_\sigma\rar\wti{V}_\sigma$.

Let $(C,\vc{p},\ph)\in\wti{V}_\sigma$ and let $\sigma(i)=j$.
Because both $p_i,p_{i_j}$ belong to the disk $\wti{\disk}^j_\sigma(\ph)$,
there exists a $\ph$-geodesic between $p_i$ and $p_{i_j}$ contained
inside $\wti{\disk}^j_\sigma(\ph)$, which is thus a concatenation of inner segments,
by Remark \ref{rmk:short-geodesics}.
Hence, $\ph$-segments
contained in $\wti{\disk}_\sigma(\ph)$ generate
$H_1(\wti{\disk}_\sigma(\ph),\vc{p};\RR)$.

Fix now $(C,\vc{p},\ph)\in \wti{V}_\sigma$.
Segments completely contained inside $\wti{\disk}^j_\sigma(\ph)$ have length at most $2r_\sigma(\ph)=c\cdot R_\sigma(\ph)$.
Thus, it is sufficient to show that, for every $(C,\vc{p},\ph)\in \wti{V}_\sigma$,
a $\ph$-segment $\gamma$ not contained in
any $\wti{\disk}^j_\sigma(\ph)$ satisfies $\frac{2r_\sigma(\ph)}{\ell_{\gamma}(\ph)}<c$, that is $\frac{R_\sigma(\ph)}{\ell_{\gamma}(\ph)}<1$.

Let $p_i$ and $p_{i'}$ be the endpoints of such a $\gamma$ (the case $i=i'$ not being excluded). There are two cases.

\begin{itemize}
\item[(a)]
Suppose $\sigma(i)=\sigma(i')=j$.\\
Let $\beta$ the geodesic from $p_{i'}$ to $p_i$
contained inside $\wti{\disk}^j_\sigma(\ph)$.
Since $\gamma\neq\beta$, the concatenation
$\gamma\ast\beta$ is a nontrivial closed loop in $C$.
Thus,
$\ell_{\gamma}(\ph)+\ell_{\beta}(\ph)\geq 2R_\sigma(\ph)$,
which implies 
\[
\ell_\gamma(\ph)\geq 2R_\sigma(\ph)-\ell_\beta(\ph)\geq 2R_{\sigma}(\ph)-2r_{\sigma}(\ph)=(2-c)R_\sigma(\ph)>R_\sigma(\ph)
\]
and so we are done.
\item[(b)]
Suppose that $\sigma(i)=j\neq j'=\sigma(i')$.\\
Since $p_i$ belongs to the disk $\wti{\disk}^j_\sigma(\ph)$ centered at $p_{\bar{\i}}$
with $\bar{\i}=i_j$
and $p_{i'}$ belongs to the disk $\wti{\disk}^{j'}_\sigma(\ph)$ centered at 
$p_{\bar{\i}'}$ with $\bar{\i}'=i_{j'}$,
we have
\begin{align*}
\ell_\gamma(\ph) & \geq dist_{\ph}(p_i,p_{i'})\geq dist_{\ph}(p_{\bar{\i}},p_{\bar{\i}'})-
dist_{\ph}(p_{\bar{\i}},p_{i})-dist_{\ph}(p_{\bar{\i}'},p_{i'})
\geq \\
& \geq 2R_\sigma(\ph)-2r_\sigma(\ph)=
(2-c)R_\sigma(\ph)>R_\sigma(\ph)
\end{align*}
because the two disks are disjoint. Thus, again we are done.
\end{itemize}
\end{proof}

Finally, we have achieved our goal.

\begin{proof}[Proof of Proposition \ref{prop:good}]
By Lemma \ref{lemma:adapted}, the open cover
$\Vfrak=\{V_\sigma\}$ is adapted to the stratification
of $\PP\Omega\M'_{g,n}$ and the disk bundle $\disk^j_\sigma\rar V_\sigma$
endows $V_\sigma$ with a $c$-merging datum for every $\sigma\in\Sfrak$
by Lemma \ref{lemma:clashing}.
Thus, in order to have a good cover it is enough to choose
$c\leq c_0(g,n)$.
\end{proof}


\subsection{Geometric functions on thickened strata}

Let $\Vfrak$ be a good cover of $\PP\Omega\M'_{g,n}$, whose existence is
granted by Proposition \ref{prop:good}, and let $\wti{\Vfrak}$ the pull-back
cover of $\Omega\M'_{g,n}$.
Fix a $\sigma\in\Sfrak$ and a $c$-merging datum $\{\disk^j_\sigma\rar V_\sigma\}$ for $V_\sigma$.

%
%
%


\begin{notation}
For every $(C,\vc{p},\ph)\in\wti{V}_\sigma$ or $(C,\vc{p},[\ph])\in V_\sigma$,
we will denote by $\Hbb_1(\ul{\RR}_{\disk_\sigma(\ph),\vc{p}})$ the $\RR$-dual
of $\Hbb^1(\ul{\RR}_{\disk_\sigma(\ph),\vc{p}})$. Clearly, if the points
$p_i$ are distinct, then $\Hbb_1(\ul{\RR}_{\disk_\sigma(\ph),\vc{p}};\RR)$
coincides with $H_1(\disk_\sigma(\ph),\vc{p};\RR)$.
\end{notation}

\begin{definition}
A {\it{$\sigma$-inner basis}} for $(C,\vc{p},[\ph])\in V_\sigma$
is a collection $B^{inn}$ of arcs contained in $\disk_\sigma(\ph)$,
whose classes form a basis
of $\Hbb_1(\ul{\RR}_{\disk_\sigma(\ph),\vc{p}})$. A {\it{$\sigma$-outer basis}}
for $(C,\vc{p},[\ph])$ is a collection $B^{out}$ of arcs, whose classes form a basis
of $H_1(C,\disk_\sigma(\ph);\RR)$.
\end{definition}

\begin{center}
\begin{figurehere}
\psfrag{C}{$C$}
\psfrag{N1}{$\disk^1_\sigma$}
\psfrag{N2}{$\disk^2_\sigma$}
\psfrag{N3}{$\disk^3_\sigma$}
\psfrag{p1}{$p_1$}
\psfrag{p2}{$p_2$}
\psfrag{p3}{$p_3$}
\psfrag{p4}{$p_4$}
\psfrag{p5}{$p_5$}
\psfrag{b}{$\beta$}
\psfrag{g}{$\gamma$}
\includegraphics[width=\textwidth]{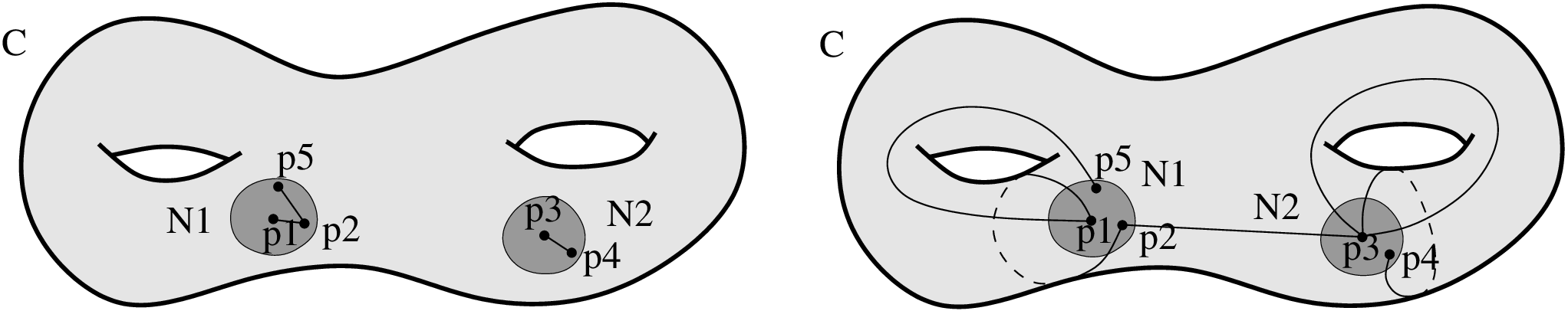}
\caption{{\small Example of inner (on the left) and outer (on the right) bases.}}\label{fig:bases}
\end{figurehere}
\end{center}
\medskip

\begin{remark}
The vector spaces $\Hbb_1(\ul{\RR}_{\disk_\sigma(\ph),\vc{p}})$ and $H_1(C,\disk_\sigma(\ph);\RR)$ determine flat $\RR$-vector bundles $\Hloc_1(\ul{\RR}_{\disk_\sigma,\vc{P}})$ and $\Hloc_1(\ul{\RR}_{\C,\disk_\sigma})$
on $V_\sigma$ and on $\wti{V}_\sigma$. 
Thus, the collections $\Bcal_\sigma^{inn}$ of all $\sigma$-inner bases and
$\Bcal_\sigma^{out}$ of all $\sigma$-outer bases determine locally constant sheaves of sets on $\wti{V}_\sigma$ and on $V_\sigma$, which will be denoted by the same symbols.
We also denote by $\Bcal_\sigma^{inn}(\ph)$ and $\Bcal_\sigma^{out}(\ph)$ their stalks at $(C,\vc{p},\ph)\in \wti{V}_\sigma$ (and at $(C,\vc{p},[\ph])\in V_\sigma$).
\end{remark}

Let $B^{inn}$ be a $\sigma$-inner basis and $B^{out}$ be a $\sigma$-outer basis
for $(C,\vc{p},\ph)\in \wti{V}_\sigma$ and let $U$ be a contractible neighbourhood
of $(C,\vc{p},\ph)$ in $\wti{V}_\sigma$ on which the universal family and the above 
locally constant sheaves $\Bcal_\sigma^{out}$ and $\Bcal_\sigma^{inn}$ can thus be trivialized.
Then we can define the local functions
$|\Pcal_{B^{inn}}|^2,\ell^{-2}_{B^{out}}:U\rar\RR$  as
\[
|\Pcal_{B^{inn}}|^2(\ph):=\sum_{\beta\in B^{inn}}|\Pcal_\beta(\ph)|^2,
\qquad\qquad
\ell^{-2}_{B^{out}}(\ph):=\sum_{\gamma\in B^{out}}\ell^{-2}_\gamma(\ph)
\]

As in Section \ref{sec:geometric}, taking suprema over all
$B^{inn}\in\Bcal^{inn}_\sigma(\ph)$
and
$B^{out}\in\Bcal^{out}_\sigma(\ph)$
determines global functions
\[
|\Pcal_{\Bcal_\sigma^{inn}}|^2=\sup_{B^{inn}}|\Pcal_{B^{inn}}|^2,
\qquad\qquad
\ell^{-2}_{\Bcal_\sigma^{out}}=\sup_{B^{out}}\ell^{-2}_{B^{out}}
\]
on $\wti{V}_\sigma$. 
%
%
%
%
%
The above functions
can be combined to define
$\eta_\sigma,\zeta_\sigma:V_\sigma\rar\RR$ as
\[
\begin{array}{rlcl}
\dis\eta_\sigma & :=
A\cdot \ell^{-2}_{\Bcal_\sigma^{out}}=
\dis\sup_{B^{out}} 
 \eta_{\sigma,B^{out}}
&\phantom{XXXX}
&
\dis
\eta_{\sigma,B^{out}}=A\cdot\ell_{B^{out}}^{-2}\\
\dis\zeta_\sigma 
& 
\dis
:=
|\Pcal_{\Bcal_\sigma^{inn}}|^2\cdot\ell^{-2}_{\Bcal_\sigma^{out}}=
\sup_{B=B^{inn}\cup B^{out}}
\zeta_{\sigma,B}
&&
\dis
\zeta_{\sigma,B}=|\Pcal_{B^{inn}}|^2\cdot \ell^{-2}_{B^{out}}
\end{array}
\]
%

Here we collect a few simple properties of the above functions.

\begin{lemma}\label{lemma:relative}
Let $(C,\vc{p},\ph)\in \wti{V}_\sigma$. Then
\begin{itemize}
\item[(a)]
there exists a contractible neighbourhood $U'\subset\wti{V}_\sigma$ of
$(C,\vc{p},\ph)$ and finitely many outer bases $\{B^{out}_j\}$ for $\ph$
such that for every $(C',\vc{p}',\ph')\in U'$ the following holds:
$\ell^{-2}_{\Bcal_\sigma^{out}}(\ph')$
is attained at some $B^{out}_j$
and all the outer bases $\{B^{out}_j\}$ are realized on $(C',\vc{p}')$ by
proper $\ph'$-segments;
\item[(b)]
the function $\ell^{-2}_{\Bcal_\sigma^{out}}$ is continuous and its
restriction to the locally closed stratum $\Omega\M'_{g,n}(\sigma)$
agrees with the
$\ell^{-2}_{\Bcal}$ defined on $\Omega\M'_{g,n}(\vc{m})$
with $\vc{m}=\sigma_*\vc{\mo}$
in Section \ref{sec:length};
\item[(c)]
$\Bcal^{inn}_\sigma(\ph)$ is a finite set;
\item[(d)]
if the value of $\zeta_\sigma$ is attained at $B=B^{inn}\cup B^{out}$, then
the value of $\eta_\sigma$ is attained at $B^{out}$;
\end{itemize}
\end{lemma}
\begin{proof}
Part (a) is similar as in Lemma \ref{lemma:ell-properties},
replacing bases by outer bases. This easily implies that
$\ell^{-2}_{\Bcal_\sigma^{out}}$ is continuous.
The remaining part of (b) is clear, since the definitions of the two
functions essentially coincide on the locally closed stratum.

Part (c) is immediate, since there are finitely many marked points,
and part (d) easily follows after rewriting $\zeta_\sigma$ as
$\left(|\Pcal_{\Bcal_\sigma^{inn}}|^2/A\right)\eta_\sigma$.
\end{proof}

The following technical result clarifies how to combine
$\eta_\sigma$ and $\zeta_\sigma$ to obtain a function $\exh_\sigma$
that satisfies condition ($\star$).

\begin{lemma}\label{lemma:chi}
Let $\a>0$ and let $\chi:[0,\a)\rar\RR_+$ be smooth function with $\chi(0)=1$ such that $\chi'>0$ and $(\log\chi)''>0$.
Then the function $\exh_\sigma:W_\sigma\rar\RR$ defined as
\[
\exh_\sigma:=\log(\eta_\sigma+\chi\circ\zeta_\sigma)
\]
satisfies condition $(\star)$, where $W_\sigma=\{\zeta_\sigma<\a\}\subseteq V_\sigma$.
\end{lemma}

Clearly, for every $B=B^{inn}\cup B^{out}$ with
$B^{inn}\in\Bcal^{inn}_\sigma(\ph)$ and
$B^{out}\in \Bcal^{out}_\sigma(\ph)$,
we can locally
define $\exh_{\sigma,B}:=\log(\eta_{\sigma,B^{out}}+
\chi\circ\zeta_{\sigma,B})$. 
By Lemma \ref{lemma:relative}, we deduce that
$\exh_\sigma=\sup_{B} \exh_{\sigma,B}$.

\begin{proof}
Let  $(C,\vc{p},\ph)\in \wti{V}_\sigma$.
As taking the sup
can only
improve subharmonicity, 
in order to check condition $(\star)$ at $(C,\vc{p},\ph)$
it is enough to work with
a fixed basis $B=B^{inn}\cup B^{out}$ such
that $\eta_\sigma(\ph)=\eta_{\sigma,B^{out}}(\ph)$ and
$\zeta_\sigma(\ph)=\zeta_{\sigma,B}(\ph)$.
Because arcs in $B^{out}$ are realized by proper $\ph$-segments
(see Lemma \ref{lemma:relative}), the functions
$\eta_{\sigma,B^{out}}$ and $\zeta_{\sigma,B}$ are
smooth near $(C,\vc{p},\ph)$. Hence, it is sufficient to compute the complex Hessian of $\exh_{\sigma,B}$.

As usual, we can work in period coordinates 
in a small open neighbourhood of $(C,\vc{p},\ph)$
by Proposition
\ref{prop:period-closed-hol}.
Let $\Pi_\ph\subset \Hbb^1(\ul{\CC}_{C,\vc{p}})$ be the kernel of the projection onto $H^{0,1}_\ph(C)$ and let $\Pi'_\ph\subset\Pi_\ph$
be a complement of the line $\CC(\ph)$.
The conclusion will follow if we prove that the restriction of $i\pa\ol{\pa}\log(\exh_{\sigma,B})_\ph$ to $\Pi'_\ph$ is positive-definite.
 
Consider a deformation $(\ph_\e)=(\ph+\e\dot{\ph})$ for a small $\e\in\CC$,
with $0\neq (\dot{\ph})\in\Pi'_\ph$.
We want to show that both $i\pa\ol{\pa}\log(\eta_{\sigma,B^{out}})_\ph(\dot{\ph},\dot{\ph})\geq 0$ and $i\pa\ol{\pa}\log(\chi\circ\zeta_{\sigma,B})_\ph(\dot{\ph},\dot{\ph})\geq 0$
and in fact at least one of the two is strictly positive.
By Lemma \ref{lemma:star-properties}(e), we will then be able to conclude that $i\pa\ol{\pa}\log(\exh_{\sigma,B})_\ph(\dot{\ph},\dot{\ph})>0$.

A computation analogous to the ones in Lemma \ref{lemma:A}(c)
and in Lemma \ref{lemma:ell-properties2}(c)
shows that
\mbox{$i\pa\ol{\pa}\log(\eta_{\sigma,B^{out}})_\ph(\dot{\ph},\dot{\ph})\geq 0$}.
On the other hand,
\[
i\pa\ol{\pa}\log(\chi\circ\zeta_{\sigma,B})=i
\frac{[(\chi\chi')\circ\zeta_{\sigma,B}] \pa\ol{\pa}\zeta_{\sigma,B}+
\left[
(\chi\chi''-(\chi')^2)\circ\zeta_{\sigma,B})
\right]\pa\zeta_{\sigma,B}\wedge\ol{\pa}\zeta_{\sigma,B}
}{(\chi\circ\zeta_{\sigma,B})^2}
\]
As $\chi,\chi'>0$ and $\chi\chi''-(\chi')^2=(\log\chi)''>0$,
we certainly have \mbox{$i\pa\ol{\pa}\log(\chi\circ\zeta_{\sigma,B})_\ph
(\dot{\ph},\dot{\ph})\geq 0$}.

Suppose now by contradiction that both the following conditions are satisfied:
\begin{itemize}
\item[(a)]
\mbox{$i\pa\ol{\pa}\log(\eta_{\sigma,B^{out}})_\ph(\dot{\ph},\dot{\ph})=0$}
\item[(b)]
\mbox{$i\pa\ol{\pa}\log(\chi\circ\zeta_{\sigma,B})_\ph
(\dot{\ph},\dot{\ph})=0$}.
\end{itemize}
%
Condition (a) necessarily implies that
$i\pa\ol{\pa}\log(A)_\ph(\dot{\ph},\dot{\ph})=0$ and $i\pa\ol{\pa}\log
(\ell_{B^{out}}^{-2})_\ph(\dot{\ph},\dot{\ph})=0$. The first equality says that
$0=(\dot{\ph})\in H^1(C;\CC)$ and so
$\e \mapsto A(\ph_\e)$ must necessarily be constant;
the second equality says that
there must be a constant $\lambda\in\CC$
such that $\int_\gamma \dot{\ph}=\lambda\int_\gamma \ph$ for all
$\gamma\in B^{out}$. As a consequence,
$\eta_{\sigma,B^{out}}([\ph_\e])=
\eta_{\sigma,B^{out}}([\ph])\cdot |1+\e \lambda|^{-2}$ and so
\[
\zeta_{\sigma,B}([\ph_\e])=
\frac{\eta_{\sigma,B^{out}}([\ph])}{A(\ph)}\,
\frac{|\Pcal_{B^{inn}}(\ph_\e)|^2}{|1+\e \lambda|^2}
\]
Condition (b) implies that
$(i\pa\ol{\pa}\zeta_{\sigma,B})_\ph
(\dot{\ph},\dot{\ph})=0$ and so
\[
\sum_{\beta\in B^{inn}}i\pa_\e\ol{\pa}_{\ol{\e}}\left|\frac{\Pcal_\beta(\ph_\e)}{1+\e\lambda} \right|^2=0.
\]
Since $\e\mapsto (1+\e\lambda)^{-1}\Pcal_\beta(\ph_\e)$ is holomorphic,
each summand must be constant in $\e$, namely
$\Pcal_\beta(\ph_\e)=(1+\lambda\e)\Pcal_\beta(\ph)$,
or equivalently
$\int_\beta \dot{\ph}=\lambda\int_\beta \ph$,
for all $\beta\in B^{inn}$.

Being $B=B^{inn}\cup B^{out}$ a set of generators for $H_1(C,\vc{p};\RR)$,
we conclude that $(\dot{\ph})=\lambda(\ph)$, thus reaching
the desired contradiction.
\end{proof}

Given an open subset $W_\sigma=\{\zeta_\sigma<\a\}\subseteq V_\sigma$ as in the previous lemma,
the non-negative $\exh_\sigma:W_\sigma\rar \RR$ will not necessarily be proper: it will depend on
the value $\a$ and on the function $\chi$.
In order to make the right choice,
we need to examine the behavior of $\eta_\sigma$ and $\zeta_\sigma$
along diverging sequences in $V_\sigma$.

\begin{lemma}[Diverging sequences in $V_\sigma$]\label{lemma:e}
Let $\left\{(C_s,\vc{p}_s,[\ph_s])\right\}_{s\in\NN}\subset V_\sigma$ be a 
diverging sequence. Up to extracting a subsequence, one of the following holds:
\begin{itemize}
\item[(a)]
the sequence diverges in $\PP\Omega\M'_{g,n}$ and
$\eta_\sigma(C_s,\vc{p}_s,[\ph_s])\rar+\infty$;
\item[(b)]
the sequence converges to a $(C,\vc{p},[\ph])\in\pa V_\sigma$ lying in a deeper stratum
and $\eta_\sigma(C_s,\vc{p}_s,[\ph_s])\rar+\infty$;
\item[(c)]
the sequence converges to a $(C,\vc{p},[\ph])\in \pa V_\sigma$
that does not lie in a deeper stratum and
$\liminf_{s\rar\infty} \zeta_\sigma(C_s,\vc{p}_s,[\ph_s])\geq a(g,n,c)$,
where $a(g,n,c)>0$ is a constant that depends only on $g$, $n$
and the $c$ chosen in Proposition \ref{prop:good}.
\end{itemize}
\end{lemma}
\begin{proof}
In cases (a) and (b), 
the argument is the same
as in the proof of Proposition \ref{prop:function}. So we can focus on case (c).

Since $(C,\vc{p},[\ph])$ belongs to $\pa V_\sigma$
but not to a stratum deeper than $\sigma$, 
(up to subsequences)
there exist two distinct indices $i,i'\in\{1,\dots,n+2g-2\}$ such that
$\sigma(i)=\sigma(i')=j$, $i'=i_j$ and $dist_{\ph_s}(p_i,p_{i'})\rar
r_\sigma(\ph_s)>0$.
Now, the $\ph_s$-shortest geodesic between $p_i$ and $p_{i'}$ is a concatenation of at most $2g-3+n$ proper segments.
Thus, there exists at least one such segment $\beta_s$
inside $\disk^j_\sigma(\ph_s)$ such that
\[
|\Pcal_{\beta_s}(\ph_s)|=
\ell_{\beta_s}(\ph_s)\geq 
\frac{r_\sigma(\ph_s)}{2g-3+n}=
\frac{c\cdot R_\sigma(\ph_s)}{2(2g-3+n)}
\]
Notice that the minimum $\ph_s$-length of a $\ph_s$-outer segment is $2R_\sigma(\ph_s)$. 
Thus, for every outer $\ph_s$-segment $\gamma_s\subset C_s$, we would have
\[
\frac{|\Pcal_{\beta_s}(\ph_s)|^2}{\ell^2_{\gamma_s}(\ph_s)}\geq
\frac{|\Pcal_{\beta_s}(\ph_s)|^2}{(2R_\sigma(\ph_s))^2}\geq
\frac{c^2}{16(2g-3+n)^2}
\]
Thus, $\liminf_{s\rar\infty} \zeta_\sigma(C_s,\vc{p}_s,[\ph_s])\geq a(g,n,c)$, where
$a(g,n,c):=
\frac{c^2}{16(2g-3+n)^2}$, which proves the claim.
\end{proof}

Let $\a<a(g,n,c)$ be a positive constant
where $a(g,n,c)>0$ is in Lemma \ref{lemma:e}:
for instance, we can choose $\a=a(g,n,c)/2$.

\begin{definition}\label{def:W}
Let $\displaystyle W_\sigma
:=\left\{(C,\vc{p},[\ph])\in V_\sigma\ \big|\ \zeta_\sigma(C,\vc{p},[\ph])<\a\right\}$
for every $\sigma\in\Sfrak$.
\end{definition}
%
\begin{remark}\label{rmk:good-W}
For every $\sigma\in\Sfrak$, the subset $W_\sigma \subset V_\sigma\subset \PP\Omega\M'_{g,n}$ is open and contains the stratum
$\PP\Omega\M'_{g,n}(\sigma)$, since $\PP\Omega\M'_{g,n}(\sigma)\subset V_\sigma$
and $\zeta_\sigma$ is a continuous function on $V_\sigma$ that vanishes on such stratum.
Thus $\Wfrak=\{W_\sigma\,|\,\sigma\in\Sfrak\}$ is
a refinement of $\Vfrak$ that satisfies property (AS1) and so it is a good open cover
by Lemma \ref{lemma:refinement}.
\end{remark}

As a consequence of Lemma \ref{lemma:e}, we obtain information on the behavior of
diverging sequence in the constructed thickening $W_\sigma$
of the stratum $\PP\Omega\M'_{g,n}(\sigma)$.

\begin{corollary}[Diverging sequences in $W_\sigma$]\label{cor:diverging-W}
Let $\{(C_s,\vc{p}_s,[\ph_s])\}_{s\in\NN}\subset W_\sigma$ be a diverging sequence.
Up to extracting a subsequence, one of the following holds:
\begin{itemize}
\item[(a)]
the sequence $(C_s,\vc{p}_s,[\ph_s])$ diverges in $\PP\Omega\M'_{g,n}$, and so
$\eta_\sigma(C_s,\vc{p}_s,[\ph_s])\rar +\infty$;
\item[(b)]
the sequence $(C_s,\vc{p}_s,[\ph_s])$ converges to $(C,\vc{p},[\ph])\in \pa W_\sigma$
lying in a deeper stratum, and so $\eta_\sigma(C_s,\vc{p}_s,[\ph_s])\rar+\infty$;
\item[(c)]
the sequence $(C_s,\vc{p}_s,[\ph_s])$ converges to $(C,\vc{p},[\ph])\in V_\sigma\cap \pa W_\sigma$
at which $\zeta_\sigma(C,\vc{p},[\ph])=\a$.
\end{itemize}
\end{corollary}
\begin{proof}
Assume that 
we are not in cases (a) or (b). Up to subsequences, $(C_s,\vc{p}_s,[\ph_s])$
converges
to a $(C,\vc{p},[\ph])\in \pa W_\sigma$ which does not lie in a deeper stratum.
Clearly, $(C,\vc{p},[\ph])$ cannot lie in $\pa V_\sigma$, for we would have
$\liminf_{s\rar\infty}\zeta_\sigma(C_s,\vc{p}_s,[\ph_s])\geq a(g,n,c)>\a$
by Lemma \ref{lemma:e}(c). Thus, $(C,\vc{p},[\ph])\in V_\sigma \cap \pa W_\sigma$, which implies
that $\zeta_\sigma(C,\vc{p},[\ph])=\a$.
\end{proof}

Here is the final fruit of our analysis.

\begin{corollary}\label{cor:star-W}
Consider the function $\chi:[0,\a)\rar\RR_+$ defined as $\chi(t)=\frac{\a}{\a-t}$.
Then $\exh_\sigma:=\log(\eta_\sigma+\chi\circ\zeta_\sigma):W_\sigma\rar[0,+\infty)$
is an exhaustion function on $W_\sigma$ that satisfies condition $(\star)$.
\end{corollary}
\begin{proof}
The chosen $\chi$ satisfies the hypotheses of Lemma \ref{lemma:chi}
and so such $\exh_\sigma$ satisfies condition $(\star)$.
Moreover, $\exh_\sigma$ is proper by Corollary \ref{cor:diverging-W} and so it is
an exhaustion function.
\end{proof}

Clearly, many other functions $\chi$ would work: indeed, it is enough to require
$\chi(0)=1$, $\chi'>0$, $(\log\chi)''>0$ and $\lim_{t\rar \a}\chi(t)=+\infty$.


\subsection{Cohomological dimension of $\M_{g,n}$}\label{sec:final}

We wish to estimate the cohomological dimension by
invoking Corollary \ref{cor:stratification-Dolbeault}.
In order to do so,
we produce an open cover of 
the moduli space $\PP\Omega\M'_{g,n}$ 
of projective Abelian differentials with marked zeros,
whose
open sets carry exhaustion functions which are 
somehow compatible with each other.

\begin{proposition}[Good $(g+1)$-complete covers]\label{prop:star}
There exists a good open cover $\Wfrak=\{W_\sigma\}$
of the moduli space $\PP\Omega\M'_{g,n}$ of Abelian differentials
with marked zeros and exhaustions
functions $\exh_\sigma: W_\sigma\rar\RR$
that satisfy condition ($\star$) for all $\sigma\in\Sfrak$.
\end{proposition}

\begin{proof}
Consider the open subsets $W_\sigma$ as in Definition \ref{def:W}.
The collection $\Wfrak=\{W_\sigma\}$ is a good open cover
of $\PP\Omega\M'_{g,n}$ by Remark \ref{rmk:good-W}. Moreover,
Corollary \ref{cor:star-W} ensures that each $W_\sigma$
is endowed with an exhaustion function
$\exh_\sigma:W_\sigma\rar\RR$ 
that satisfies condition $(\star)$.
\end{proof}

This immediately leads to
the desired conclusions.

\begin{proof}[Proof of Theorem \ref{mthm:closed}]
Let $\sigma_\bullet$ be a simplex in the barycentric subdivision $\Sfrak'_\tau$ of
$\Sfrak_\tau$ (see Section \ref{subsection:coverings}).

A point $(C,\vc{p},[\ph])\in W_{\sigma_\bullet}\cap \PP\Omega\M'_{g,n}(\ol{\tau})$
represents a surface with at most $n+k$ distinct zeros.
Since $p_1,\dots,p_n$ must be distinct, it follows that $\dim(\sigma_\bullet)\leq k-1+\eps_n$ and so $\dim(\Sfrak')\leq k-1+\eps_n$.
(Indeed, it can be shown that $\dim(\Sfrak')=k-1+\eps_n$ but we will not need it.)

Now, the function
\[
\exh_{\sigma_\bullet}: W_{\sigma_\bullet}\rar\RR
\]
defined as $\exh_{\sigma_\bullet}:=\exh_{\sigma_0}+\dots+\exh_{\sigma_d}$
is clearly an exhaustion function.
Since all $\exh_{\sigma_i}$ satisfy
condition ($\star$) due to Proposition \ref{prop:star},
so does their sum $\exh_{\sigma_\bullet}$ by
Lemma \ref{lemma:star-properties}(d). Thus, $\exh_{\sigma_\bullet}$ is $(g+1)$-convex.
Hence, Corollary \ref{cor:stratification-Dolbeault} applied to $\Wfrak$
implies that
$\cohdim_{Dol}(\PP\Omega\M'_{g,n})\leq (k-1+\eps_n)+g$.
\end{proof}

\begin{proof}[Proof of Theorem \ref{mthm:hodge}]
By Lemma \ref{lemma:fibr}(a) applied to the finite map
$\PP\Omega\M'_{g,n}\rar \PP\Omega\M_{g,n}$
that forgets the last $2g-2$ markings,
we know that $\PP\Omega\M_{g,n}$ and $\PP\Omega\M'_{g,n}$
have the same $\cohdim_{Dol}$.

Hence, Theorem \ref{mthm:hodge} reduces to a special case of Theorem \ref{mthm:closed},
since $\PP\Omega\M'_{g,n}(\sigma)$ is dense in $\PP\Omega\M'_{g,n}$
for $k=2g-2$ and $\sigma=id$.
\end{proof}

Finally, we have achieved our main result.

\begin{proof}[Proof of Theorem \ref{mthm:cohdim}]
Consider the projective bundle $\PP\Omega\M_{g,n}\rar \M_{g,n}$
with $(g-1)$-dimensional fibers
that forgets the projective Abelian differential.
By Corollary \ref{cor:fibration}, we have that
\[
\cohdim_{Dol}(\PP\Omega\M_{g,n})=(g-1)+\cohdim_{Dol}(\M_{g,n})\, .
\]
Since $\cohdim_{Dol}(\PP\Omega\M_{g,n})\leq (2g-3+\eps_n)+g$
by Theorem \ref{mthm:hodge},
we conclude that $\cohdim_{Dol}(\M_{g,n})\leq (g-2+\eps_n)+g$.
\end{proof}

\appendix

\section{Cohomological dimension}\label{sec:appendix}

\subsection{De Rham cohomology of orbifolds}\label{subsection:orbifold}

Let $X$ be a smooth orbifold.
Locally $X$ looks like $[\wti{U}/G]$, where $\wti{U}$ is an open subset of
a Euclidean space and $G$ is a finite group acting on $\wti{U}$.
In particular, change of charts
$\wti{U}_i \leftarrow \wti{U}_{ij}\rar \wti{U}_j$
are \'etale and so local diffeomorphisms.

By definition, smooth functions on $[\wti{U}/G]$ are $G$-invariant smooth functions on $\wti{U}$; analogously, one can define the sheaves $\mathcal{A}^q_X$ of smooth
differential $q$-forms on $X$.

The correspondence between locally free
$\mathcal{A}^0_X$-modules and smooth vector bundles on $X$ carries on, and
affine connections are defined in the usual way.

Thus, one can speak of de Rham cohomology
$H^*_{dR}(X;\mathbb{L})$ of the orbifold $X$ with coefficients in
a flat vector bundle $\mathbb{L}$ as the cohomology of the complex
\[
0\rar
(\mathcal{A}^0_X\otimes \mathbb{L})(X)
\arr{d}{\lra}
(\mathcal{A}^1_X\otimes\mathbb{L})(X)
\arr{d}{\lra}
(\mathcal{A}^2_X\otimes\mathbb{L})(X)
\arr{d}{\lra} \dots
\]
\begin{notation}
We will constantly identify a flat vector bundle and its associated local system
(that is, the sheaf of its parallel sections) and indeed we will denote them by
the same symbol.
\end{notation}

The case of an orbifold $X=[\wti{X}/G]$ which is a global quotient
of a manifold $\wti{X}$ by a finite group $G$ is rather special and easier to
deal with.
Indeed, if $\rho:\wti{X}\rar X$ is the quotient map and $\mathbb{L}\rar X$
is a flat vector bundle, 
then $\rho^*\mathbb{L}\rar\wti{X}$ is
a $G$-equivariant flat vector bundle
and
$H^q_{dR}(X;\mathbb{L})=
H^q_{dR}(\wti{X};\rho^*\mathbb{L})^{G}$.
Moreover, for any flat vector bundle $\wti{\mathbb{L}}\rar\wti{X}$,
the push-forward $\rho_*\wti{\mathbb{L}}\rar X$ is also a flat vector bundle
because $\rho$ is finite \'etale (in the orbifold sense) and surjective; 
moreover,
$H^q_{dR}(\wti{X};\wti{\mathbb{L}})=H^q_{dR}(X;\rho_*\wti{\mathbb{L}})$.
Hence, de Rham cohomology theories of $X$ and $\wti{X}$ with coefficients
in flat vector bundles somehow carry the same information.

\begin{definition}
The {\it{de Rham cohomological dimension
$\cohdim_{dR}(X)$ of $X$}} is
the maximum integer $s\geq 0$ such that $H^s_{dR}(X;\mathbb{L})\neq 0$
for some flat $\CC$-vector bundle $\mathbb{L}$ on $X$.
%
\end{definition}

Obviously, $\cohdim_{dR}(X)\leq \dim_{\RR}(X)$ and the equality holds if and only if $X$ has a compact top-dimensional component (for instance, by Poincar\'e duality).
Here we collect two more properties of $\cohdim_{dR}$ without proof.

\begin{lemma}
Let $\pi:Y\rar X$ be a smooth map between connected orbifolds.
\begin{itemize}
\item[(a)]
If $\pi$ is a proper submersion with fibers of dimension $r$, then
$\cohdim_{dR}(Y)=\cohdim_{dR}(X)+r$.
\item[(b)]
If $\pi$ is the inclusion of a closed retract, then $\cohdim_{dR}(Y)\leq\cohdim_{dR}(X)$.
Moreover, equality holds if $\pi$ is the inclusion of a deformation retract.
\end{itemize}
\end{lemma}


\subsection{Dolbeault cohomology of orbifolds}\label{subsection:Dolbeault}

Let now $X$ be a complex-analytic orbifold.
Most considerations in the previous section
can be transported to the complex-analytic
world:
indeed, it makes sense
to speak of the sheaf $\mathcal{O}_X$ of holomorphic functions,
of the sheaf $\mathcal{A}^{p,q}_X$ of smooth differential $(p,q)$-forms and of 
holomorphic vector bundles on $X$. Moreover,
Dolbeault cohomology groups $H^{0,q}_{\ol{\pa}}(X;E)$ with
coefficients in the holomorphic vector bundle $E$ over $X$ 
are defined as the $q$-th cohomology group of the complex
\[
0\rar
(\mathcal{A}^{0,0}_X\otimes E)(X)
\arr{\ol\pa}{\lra}
(\mathcal{A}^{0,1}_X\otimes E)(X)
\arr{\ol\pa}{\lra}
(\mathcal{A}^{0,2}_X\otimes E)(X)
\arr{\ol\pa}{\lra} \dots
\]

Again, if $X=[\wti{X}/G]$ is a global quotient, then
$H^{0,q}_{\ol\pa}(X;E)=
H^{0,q}_{\ol\pa}(\wti{X};\rho^*E)^{G}$,
where $\rho:\wti{X}\rar X$.
Moreover, if $\wti{E}\rar\wti{X}$ is a holomorphic vector bundle, then
$\rho_*\wti{E}\rar X$ is too and
$H^{0,q}_{\ol\pa}(\wti{X};\wti{E})=H^{0,q}_{\ol\pa}(X;\rho_*\wti{E})$.

We recall that there is a natural isomorphism between
$H^{0,q}_{\ol\pa}(X;E)$ and the cohomology group $H^q(X;\mathcal{E})$
with coefficients in the locally free $\mathcal{O}_X$-module $\mathcal{E}$ of holomorphic
sections of $E$.

\begin{remark}
The push-forward of a locally free coherent sheaf through a closed embedding is not
locally free in general, but only coherent. Moreover,
analytic coherent sheaves need not admit finite locally free resolutions
(see \cite{voisin:hodge}, Corollary A.5).
\end{remark}

As a consequence of the above remark, a definition
only based on vector bundles would make it difficult to relate
the Dolbeault cohomological dimension
of a complex manifold and that of a closed holomorphic subvariety.
The above considerations motivate the following.

\begin{definition}\label{def:Dol}
The {\it{Dolbeault cohomological dimension}} $\cohdim_{Dol}(X)$ of $X$ is
the maximum integer $q\geq 0$ such that
$H^q(X;\mathcal{F})\neq 0$ for some analytic coherent sheaf $\mathcal{F}$ on $X$.
%
\end{definition}

\begin{notation}
Strictly speaking, the terminology ``Dolbeault'' is only used
for the cohomology of the complex $\big((\mathcal{A}^{0,*}\otimes E)(X),\ol{\pa}\big)$ with
$E$ a holomorphic vector bundle, whereas Definition \ref{def:Dol} involves
cohomology of analytic coherent sheaves. Though it might sound a little imprecise, we 
prefer to keep the most descriptive name ``Dolbeault cohomological dimension'' for
the invariant $\cohdim_{Dol}$.
\end{notation}

Clearly, $\cohdim_{Dol}(X)\leq \dim_\CC(X)$ and the equality is attained if and only if $X$
has a compact component of top dimension (for instance, by Serre duality).
On the opposite extreme, $X$ is Stein if and only if $\cohdim_{Dol}(X)=0$ (see \cite{serre:gaga}).

%

A relation between $\cohdim_{Dol}$ and $\cohdim_{dR}$ is provided by
the following result.
%

\begin{lemma}[Dolbeault-de Rham $\cohdim$]\label{lemma:Dol-dR}
Let $X$ be a complex-analytic orbifold. Then
\[
\cohdim_{dR}(X)\leq \cohdim_{Dol}(X)+\mathrm{dim}_{\CC}(X)\, .
\]
\end{lemma}

\begin{proof}
Let $\mathbb{L}$ be a flat complex vector bundle on $X$. Since $\mathbb{L}$ has an atlas with
locally constant transition functions, it can be also viewed as a holomorphic vector bundle.

Consider now the holomorphic de Rham complex on $X$
\[
0\rar \Ocal_X=\Omega_X^{0,0}\arr{\pa}{\lra}\Omega_X^{1,0}\arr{\pa}{\lra}
\Omega_X^{2,0}\arr{\pa}{\lra}\dots
\]
which is a resolution by holomorphic vector bundles of the locally constant sheaf $\ul{\CC}_X$.
Tensoring it by $\mathbb{L}$, we obtain
a twisted version of the Hodge-Fr\"olicher spectral sequence
\[
E_1^{p,q}=H^{0,q}_{\ol{\pa}}(X;\Omega_X^{p,0}\otimes\mathbb{L})  \implies H^{p+q}_{dR}(X;\mathbb{L})
\]
and so the conclusion follows as $E_1^{p,q}=0$ unless $p\leq\mathrm{dim}_{\CC}(X)$ and $q\leq\cohdim_{Dol}(X)$.
\end{proof}


Given a holomorphic map $\pi:Y\rar X$ of complex manifolds,
a relation between the cohomological dimensions of $Y$ and $X$
is sometimes provided by Leray spectral sequence
\[
E_2^{p,q}=H^{p}(X;R^q\pi_*\mathcal{F})\implies H^{p+q}(Y;\mathcal{F})
\]
where $\mathcal{F}$ is an analytic coherent sheaf on $Y$.

Indeed, we have the following lemma (whose hypotheses can be weakened but this is enough for our purposes).

\begin{lemma}[Fibrations]\label{lemma:fibr}
Assume $X$ and $Y$ connected. 
\begin{itemize}
\item[(a)]
If $\pi$ is finite, then $\cohdim_{Dol}(Y)\leq\cohdim_{Dol}(X)$.
Moreover, if $\pi$ is finite and surjective, then
$\cohdim_{Dol}(Y)=\cohdim_{Dol}(X)$.
\item[(b)]
If $\pi$ is submersive and with fibers of dimension $r$,
then $\cohdim_{Dol}(Y)=\cohdim_{Dol}(X)+r$.
\item[(c)]
If $\pi$ is an affine map between algebraic manifolds, then
$\cohdim_{Dol}(Y)\leq \cohdim_{Dol}(X)$.
\end{itemize}
\end{lemma}
\begin{proof}
By \cite{grauert}, if $\pi$ is proper, then
$R^q\pi_*\mathcal{F}$ is a coherent $\mathcal{O}_X$-module
for all analytic coherent sheaf $\mathcal{F}$ on $Y$ and all $q\geq 0$.

The first claim in (a)
is an immediate consequence of Leray spectral sequence above
after noticing that $R^q\pi_*\mathcal{F}=0$ for all $q>0$
and all coherent $\mathcal{O}_Y$-module $\mathcal{F}$, because $\pi$ is finite.
For the reverse inequality observe that,
if $\pi$ is finite and surjective of degree $d\geq 1$, then the trace $\mathrm{tr}_{Y/X}:\pi_*\Ocal_Y\rar\Ocal_X$ is a map of $\Ocal_X$-modules
and the composition $\Ocal_X\rar \pi_*\Ocal_Y\arr{\mathrm{tr}}{\lra}\Ocal_X$
is the multiplication by $d$ (and so an isomorphism, as we are in
characteristic $0$).
Tensoring by a coherent $\mathcal{O}_X$-module $\mathcal{G}$, we obtain
the composition
$\mathcal{G}\rar \pi_*\pi^*\mathcal{G}\rar \mathcal{G}$, which is again the multiplication
by $d$.
Hence, $H^{p}(X;\mathcal{G})\rar H^{p}(X;\pi_*\pi^*\mathcal{G})
\rar H^{p}(X;\mathcal{G})$ is also an isomorphism.
Thus, if $H^{p}(X;\mathcal{G})\neq 0$, then
$H^{p}(Y,\pi^*\mathcal{G})=H^{p}(X;\pi_*\pi^*\mathcal{G})\neq 0$, which shows that $\cohdim_{Dol}(Y)\geq \cohdim_{Dol}(X)$.

In case (b), the smoothness of $\pi$ implies
that $\mathcal{F}$ is flat over $X$, and so
$(R^q\pi_*\mathcal{F})_x\cong H^q(\pi^{-1}(x);\mathcal{F})$, which gives
$R^q\pi_*\mathcal{F}=0$ for $q>r$.
Thus, $\cohdim_{Dol}(Y)\leq r+\cohdim_{Dol}(X)$.

For the reverse inequality,
consider the invertible sheaf $\omega_{\pi}$ of vertical $(r,0)$-holomorphic forms.
It satisfies $R^q\pi_*\omega_{\pi}=0$ for $q\neq r$ and
$R^r\pi_*\omega_{\pi}\cong\Ocal_X$ by $\pi$-relative Serre duality.
Applying Leray spectral sequence to $\mathcal{F}=(\pi^*\mathcal{G})\otimes\omega_{\pi}$, and using
the projection formula $R^q\pi_*\mathcal{F}\cong \mathcal{G}\otimes R^q\pi_*\omega_\pi$, we obtain
$H^{p}(X;\mathcal{G})\cong H^{r+p}(Y;\mathcal{F})$
for every coherent $\mathcal{O}_X$-module $\mathcal{G}$, and so $\cohdim_{Dol}(Y)\geq r+\cohdim_{Dol}(X)$.

As for part (c), for every affine open subset $U\subset X$, the preimage $\pi^{-1}(U)$ is affine and so Stein. Then it is easy to see that, for every 
Stein open subset $U'\subset U$, the preimage $\pi^{-1}(U')$ is Stein too.
Hence, $H^{q}(\pi^{-1}(U');\mathcal{F})=0$ as $q>0$ for all such $U'$ and so
$R^q \pi_*\mathcal{F}=0$ for $q>0$. The conclusion follows from Leray spectral sequence.
\end{proof}

Lemma \ref{lemma:fibr}(b) has the following immediate consequence.

\begin{corollary}[Projective bundles]\label{cor:fibration}
Let $E\rar X$ be a holomorphic vector bundle of rank $r+1$ and let
$\PP E$ be its projectivization.
Then $\cohdim_{Dol}(\PP E)=r+\cohdim_{Dol}(X)$.
\end{corollary}

\subsection{Exhaustion functions}\label{subsection:exhaustions}

Let $X$ be a complex manifold and $\exh:X\rar\RR$ be a continuous function.
We recall that $\exh$ is an {\it{exhaustion function}} if it is proper and bounded from below.

\begin{definition}
A $C^2$ function $\exh:X\rar\RR$ is {\it{strongly $(q+1)$-convex}} if 
its complex Hessian
$i\pa\ol{\pa}\exh$ has index of positivity at least $\dim_\CC(X)-q$
at each point $x\in X$.
\end{definition}

We would like to elaborate upon the following result
(see also Chapter IX, Corollary 4.11 \cite{demailly:book}).

\begin{theorem}[Andreotti-Grauert \cite{andreotti-grauert}]\label{thm:q-convex}
Let $X$ be a complex manifold that admits a strongly
$(q+1)$-convex smooth exhaustion function.
%
Then
$\cohdim_{Dol}(X)\leq q$.
\end{theorem}

We will need the following minor and standard generalization of the above definition
because we will have to deal with non-smooth functions.

\begin{definition}
A continuous function $\exh:X\rar\RR$ is {\it{strongly $(q+1)$-convex}} if, for every $x\in X$, there exists a germ of locally closed
complex submanifold $L_x\subset X$ through $x$ of codimension $q$
such that $\exh|_{L_x}$ is strongly plurisubharmonic, i.e.
for every holomorphic immersion $f:\Delta\rar L_x$ there exists $\e>0$
such that $\exh\circ f-\e |z|^2$ is subharmonic on $\Delta=\{z\in\CC\,|\,|z|<1\}$.
\end{definition}

A suitable variation of Richberg approximation theorem
\cite{richberg}, using regularized maxima as in
Theorem 5.21 in \cite{demailly:book}, yields the following.

\begin{lemma}[Smooth approximation]\label{lemma:smoothening1}
Let $X$ be a complex manifold and let $\exh:X\rar\RR$ be a strongly $(q+1)$-convex function.
%
Then, for every $\nu>0$ there exists a smooth strongly $(q+1)$-convex function $\hat{\exh}:X\rar\RR$ such that
$\|\hat{\exh}-\exh\|_\infty<\nu$.
In particular, if $\exh$ is strongly plurisubharmonic on the germ $L_x$, then
$\hat{\exh}$ can be required to be strongly plurisubharmonic
on $L_x$ for all $x\in X$.
%
\end{lemma}


The application we have in mind is the following.

\begin{corollary}[Maxima of $(q+1)$-convex functions]\label{cor:max}
Let $X$ be a complex manifold and let $\exh:X\rar\RR$ be a continuous exhaustion function.
Suppose that for every $x\in X$
\begin{itemize}
\item[(a)]
there exist an open neighbourhood $U$ of $x$ and
finitely many smooth functions $\exh_i:U\rar\RR$ such that $\exh|_U=\max\{\exh_i\}$;
\item[(b)]
there exists a complex submanifold $L_x\subset U$ through $x$ of codimension $q$
such that $\exh_i|_{L_x}$ is strongly plurisubharmonic for all $i$.
\end{itemize}
Then there exists a smooth 
strongly $(q+1)$-convex
exhaustion function $\hat{\exh}$
on $X$, which in particular is strongly plurisubharmonic on each $L_x$. Moreover, if $\exh$ is invariant under a finite group $G$ of automorphisms of $X$, then $\hat{\exh}$ can be required to be $G$-invariant.
%
\end{corollary}
\begin{proof}
By Lemma \ref{lemma:smoothening1}, the function $\hat{\exh}$
is at bounded distance from $\exh$ and so $\hat{\exh}$ is an exhaustion function. As $\hat{\exh}$ is smooth and $\hat{\exh}|_{L_x}$ is strongly plurisubharmonic, it is strictly $(q+1)$-convex.

Finally, if $\exh$ is $G$-invariant, then 
the $L_x$ can be chosen to be $G$-invariant too. So
the function $\hat{\exh}^G$
defined as
$
\hat{\exh}^G(x):=\frac{1}{|G|}\sum_{g\in G} \hat{\exh}(g\cdot x)
$
is still a smooth exhaustion function and its restriction to each $L_x$ is strongly plurisubharmonic. Moreover, $\hat{\exh}^G$ is manifestly $G$-invariant.
\end{proof}
%

\subsection{Stratifications and coverings}\label{subsection:coverings}

Let $X$ be a complex manifold and suppose that $X$ is
stratified through locally closed strata $X_\sigma$,
where $\sigma\in\Sfrak$ and $\Sfrak$ is a partially ordered
set of indices, so that $X_\sigma \subseteq\ol{X}_\tau$
if and only if $\sigma\preceq\tau$.

\begin{definition}\label{def:adapted}
An open cover $\Vfrak=\{V_\sigma\,|\,\sigma\in\Sfrak\}$ of $X$ 
is {\it{adapted to the stratification 
${\{X_\sigma\,|\,\sigma\in\Sfrak\}}$
}}
if it satisfies the following two properties:
\begin{itemize}
\item[(AS1)]
$V_\sigma$ is an open neighbourhood of $X_\sigma$ for all $\sigma\in\Sfrak$;
\item[(AS2)]
$V_\sigma\cap V_\tau\neq\emptyset$ if and only if $\sigma\preceq\tau$ or $\tau\preceq\sigma$.
\end{itemize}
\end{definition}

\begin{notation}
Consider the simplicial complex $\Sfrak'$ obtained from $\Sfrak$ by barycentric subdivision,
whose $d$-simplices are chains $\sigma_\bullet=(\sigma_0\succneqq\sigma_1\succneqq\cdots\succneqq\sigma_d)$
of elements $\sigma_i\in \Sfrak$.
For $\sigma_\bullet\in\Sfrak'$, we will denote
by $V_{\sigma_\bullet}$ the intersection $V_{\sigma_0}\cap\dots\cap V_{\sigma_d}$.
\end{notation}

\begin{center}
\begin{figurehere}
\psfrag{Vs}{$V_\sigma$}
\psfrag{Vt}{$V_\tau$}
\psfrag{Vr}{$V_\rho$}
\psfrag{X}{$X$}
\psfrag{Xs}{$X_\sigma$}
\psfrag{Xt}{$X_\tau$}
\psfrag{Xr}{$X_{\rho}$}
\includegraphics[width=0.5\textwidth]{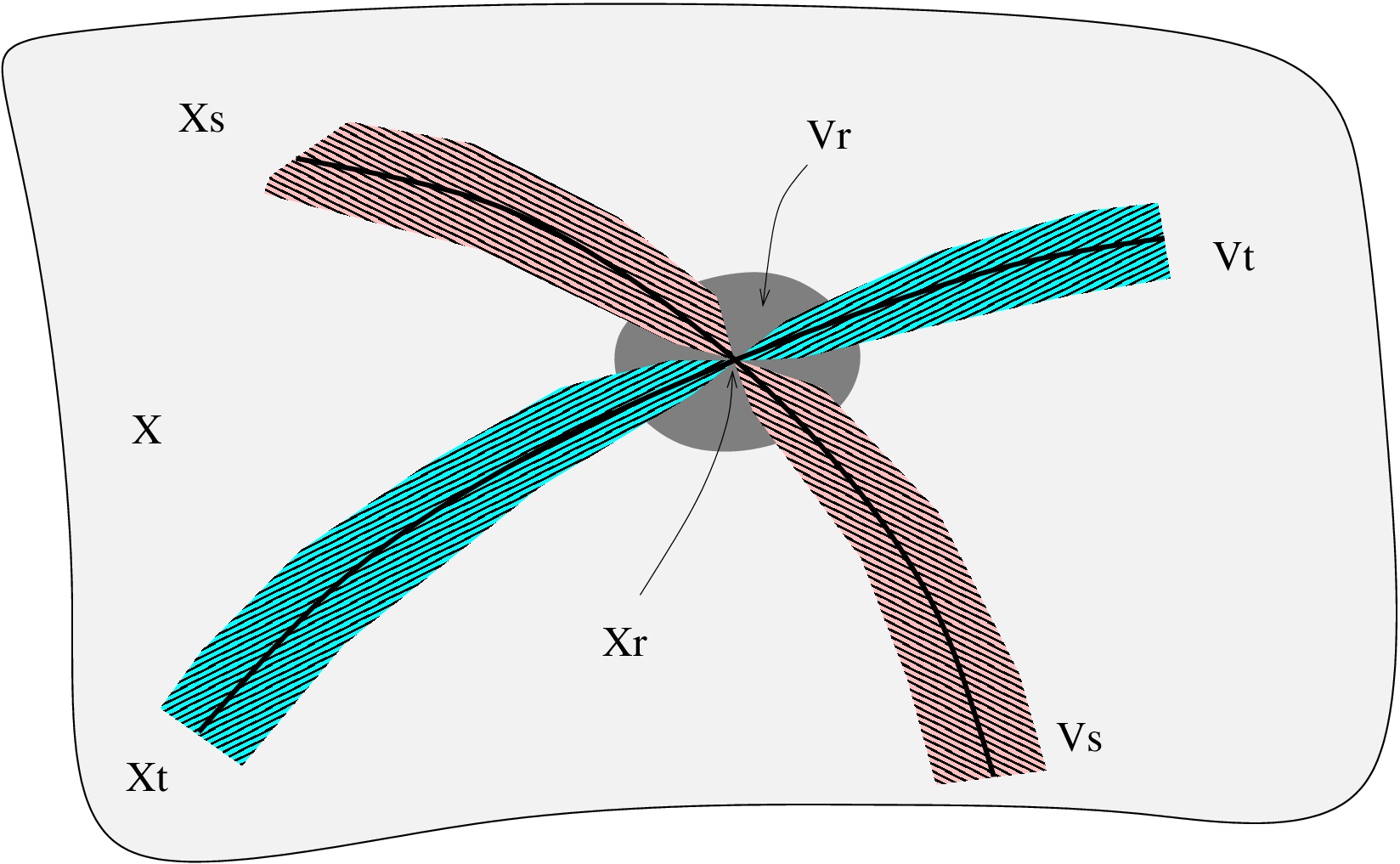}
\caption{{\small How an adapted cover looks like near a stratum $\rho\prec\sigma,\tau$.}}\label{fig:thickening}
\end{figurehere}
\end{center}

\medskip

Notice that, for every $\tau\in\Sfrak$,
the closed locus $\ol{X}_\tau$ is naturally stratified
by $X_\sigma$ with $\sigma\in\Sfrak_\tau=\{\sigma\in\Sfrak\,|\,\sigma\preceq\tau\}$.
Moreover,
the restriction $\Vfrak_\tau=\{V_\sigma\cap \ol{X}_\tau\,|\, \sigma\in\Sfrak_\tau\}$
of $\Vfrak$ is an open cover of $\ol{X}_\tau$ adapted to the induced stratification.

\medskip

We state the following for $\cohdim_{Dol}$ even though
it holds in greater generality.

\begin{lemma}[$\cohdim$ local-to-global]\label{lemma:local-global}
The Dolbeault cohomological dimension of
$X$ can be estimated as follows
\[
\cohdim_{Dol}\left(X\right) \leq
\mathrm{max}
\Big\{
\cohdim_{Dol}(V_{\sigma_\bullet})+\dim(\sigma_\bullet)
\ \Big| \ \sigma_\bullet\in\Sfrak'
\Big\}
\]
\end{lemma}
\begin{proof}
Let $\mathcal{F}$ be an analytic coherent sheaf on $X$.
By Mayer-Vietoris spectral sequence
\[
E_1^{d,q}=\bigoplus_{\sigma_0\succneqq \dots\succneqq \sigma_d} H^{q}(V_{\sigma_0}\cap\dots\cap V_{\sigma_d};\mathcal{F})
\implies
H^{d+q}(X;\mathcal{F})
\]
applied to the cover $\Vfrak$, the result immediately follows.
\end{proof}

We specialize the above estimate to the following case.

\begin{corollary}[$\cohdim$ estimate]\label{cor:stratification-Dolbeault}
Let $\Wfrak=\{W_\sigma\,|\,\sigma\in\Sfrak\}$ be an open cover of $X$
adapted to the stratification and suppose that 
$W_{\sigma_\bullet}$ is strongly $(q+1)$-convex (if non-empty) for all $\sigma_\bullet\in\Sfrak'$.
Then
\[
\cohdim_{Dol}\left(\ol{X}_\tau\right)\leq q+
\dim(\Sfrak'_\tau)
\]
for every $\tau\in\Sfrak$
and so in particular 
$\cohdim_{Dol}(X)\leq q+\dim(\Sfrak')$.
\end{corollary}
\begin{proof}
For every $\sigma_\bullet$ the intersection
$W_{\sigma_\bullet}\cap\ol{X}_\tau$ 
is a closed holomorphic submanifold
of $W_{\sigma_\bullet}$ and so
\[
\cohdim_{Dol}(W_{\sigma_\bullet}\cap\ol{X}_\tau)\leq
\cohdim_{Dol}(W_{\sigma_\bullet})\leq q
\]
by Lemma \ref{lemma:fibr}(a).
The conclusion follows from Lemma \ref{lemma:local-global}
applied to the open cover $\Wfrak_\tau$ of $\ol{X}_\tau$.
\end{proof}

\newpage

\section{List of most common symbols}

{\small
%
\begin{tabular}{ll}
$\varsigma_n$ & indicator that takes value $0$ if $n=0$ or value $1$ if $n>0$\\
$\num{n+k}$ & set $\{1,2,\dots,n+k\}$\\
$\Sfrak_n(k,l)$ & set of surjections $\num{n+k}\surj\num{n+l}$ that fix $\num{n}$\\
$\Sfrak$ & union of $\Sfrak_n(2g-2,2g-2-d)$ for $d\geq 0$\\
$\sigma'\preceq\sigma$ & the surjection $\sigma'$ 
factors through $\sigma$\\
$\Sfrak_\tau$ & set of permutations $\sigma\in\Sfrak$ such that $\sigma\preceq\tau$\\ 
$\delta_\sigma$ & image of $\M^l_{g,n}\rar\M^k_{g,n}$ associated to $\sigma\in\Sfrak_n(k,l)$\\
$\delta_{i,j}$ & Cartier divisor in $\M^k_{g,n}$ where $p_{i}=p_{j}$ for $i\neq j$\\
$\delta$ & union of all $\delta_{i,j}$\\
$\omega_\pi$ & $\pi$-vertical holomorphic cotangent bundle\\
$\Omega\M'_{g,n}$ & moduli space of non-zero Abelian differentials with marked zeros over $\M_{g,n}$\\
$\vc{m}$ & non-negative partition of $2g-2$ with $m_i>0$ for $i>n$\\
%
$\vc{\mo}$ & partition $(0^n,1^{2g-2})$\\
$\sigma_*\vc{m}$ & push-forward of $\vc{m}$ via $\sigma\in\Sfrak_n(k,l)$\\
$\Omega\M'_{g,n}(\vc{m})$ & moduli space of non-zero Abelian differentials with marked zeros\\ & of multiplicities $\vc{m}$ over $\M_{g,n}$\\
$\Omega\M'_{g,n}(\ol{\vc{m}})$ & moduli space of non-zero Abelian differentials with marked zeros\\ & of multiplicities $\vc{m}$ or more degenerate, over $\M_{g,n}$\\
$\Omega\M_{g,n}(\vc{m})$ & moduli space of non-zero Abelian differentials with unmarked zeros\\ & of multiplicities $\vc{m}$ over $\M_{g,n}$\\
$\Omega\M_{g,n}(\ol{\vc{m}})$ & moduli space of non-zero Abelian differentials with unmarked zeros\\ &of multiplicities $\vc{m}$ or more degenerate, over $\M_{g,n}$\\
$\Omega\M'_{g,n}(\sigma)$ & stratum of $\Omega\M'_{g,n}$ associated to $\sigma_*\vc{\mo}$ for $\sigma\in\Sfrak$\\
$\PP\Omega\M$ & projectivization of the moduli space $\Omega\M$\\
$\Pcal$ & (local) period map\\
$\ul{\CC}_{C,\vc{p}}$ & complex $[\ul{\CC}_C\rar\bigoplus_i \ul{\CC}_{p_i}]$ of sheaves on $C$ in degrees $[0,1]$\\
$\Pi_\varphi$ & subspace of elements in $\Hbb^1(\ul{\CC}_{C,\vc{p}})$ that map to $H^{1,0}_\varphi(C)$\\
$\Pi'_\varphi$ & a complement of $\CC(\varphi)$ inside $\Pi_\varphi$\\
$A$ & area functional\\
$\ell_\gamma(\varphi)$ & length of the geodesic homotopic to $\gamma$ for the metric $|\varphi|^2$\\
$\ell_{sys}(\varphi)$ & length of the shortest nontrivial arc of $|\varphi|^2$\\
$f_B$ & function $\sum_{\gamma\in B}(f\circ\ell_\gamma)$\\
$f_{\Bcal}$ & supremum of $f_B$ over all $B\in\Bcal$\\
$\exh_{\vc{m}}$ & exhaustion function $\log(A\cdot \ell^{-2}_\Bcal)$ on $\PP\Omega\M'_{g,n}(\vc{m})$\\
$W_\sigma\subset V_\sigma$ & open thickenings of $\PP\Omega\M'_{g,n}(\sigma)$\\
$\Wfrak$, $\Vfrak$ & open covers $\{W_\sigma\}$, $\{V_\sigma\}$ of $\PP\Omega\M'_{g,n}$\\
$\disk_{\sigma}$ & merging datum for $V_\sigma$\\
$\widetilde{V}_\sigma$ & preimage of $V_\sigma$ in $\Omega\M'_{g,n}$\\
$\widetilde{\Vfrak}$ & open cover $\{\widetilde{V}_\sigma\}$ of
$\Omega\M'_{g,n}$\\
$R_\sigma$ & injective radius function $\Omega\M'_{g,n}\rar\RR$ relative to $\sigma$\\
$\Bcal^{inn}_\sigma(\varphi)$ & set of inner bases of $H_1(\disk_\sigma(\varphi),\vc{p};\RR)$\\
$\Bcal^{out}_\sigma(\varphi)$ & set of outer bases of $H_1(C,\disk_\sigma(\varphi);\RR)$\\
$\eta_\sigma$ & function $A\cdot \ell^{-2}_{\Bcal^{out}_\sigma}$ on $V_\sigma$\\
$\zeta_\sigma$ & function $|\Pcal_{\Bcal^{inn}_\sigma}|^2\cdot\ell^{-2}_{\Bcal^{out}_\sigma}$ on $V_\sigma$\\
${\exh}_\sigma$ & function $\log(\eta_\sigma+\chi\circ\zeta_\sigma)$ on $V_\sigma$\\
$\sigma_\bullet$ & sequence $(\sigma_0\succneqq\sigma_1\succneqq\cdots\succneqq\sigma_d)$\\
$V_{\sigma_\bullet}$ & intersection $V_{\sigma_0}\cap\dots\cap V_{\sigma_d}$
\end{tabular}
}

\bibliographystyle{amsplain}
\bibliography{ab_diff-coh_dim}

\end{document}